\documentclass{lmcs}
\pdfoutput=1
\usepackage[utf8]{inputenc}

\usepackage{lastpage}
\lmcsdoi{16}{1}{28}
\lmcsheading{}{\pageref{LastPage}}{}{}%
{Apr.~11,~2018}{Feb.~28,~2020}{}

\usepackage{tikz}
\usepackage{hyperref}
\usepackage{amssymb}
\usetikzlibrary{matrix}

\usepackage{graphicx}
\usetikzlibrary{external}
\keywords{Ordinary Differential Equations, Universal Differential
  Equations, Analog Models of Computation, Continuous-Time Models of
  Computation, Computability, Computational Analysis, Computational Complexity
}

\newcommand{\R}{\mathbb{R}}
\newcommand{\Rp}{\mathbb{R}_{\geqslant0}}
\newcommand{\Rps}{\mathbb{R}_{>0}}
\newcommand{\N}{\mathbb{N}}
\newcommand{\Nps}{\mathbb{N}_{>0}}
\newcommand{\Z}{\mathbb{Z}}
\newcommand{\Q}{\mathbb{Q}}
\newcommand{\D}{\mathbb{D}}
\newcommand{\K}{\mathbb{K}}
\newcommand{\set}[1]{\left\{#1\right\}}
\newcommand{\myceil}[1]{\left\lceil#1\right\rceil}
\newcommand{\dom}{\operatorname{dom}}
\newcommand{\myfunctiontext}[1]{\texttt{#1}}
\newcommand{\myfunction}[1]{\mathtt{#1}}
\newcommand{\bitgen}{\myfunction{bitgen}}
\newcommand{\bitgentext}{\myfunctiontext{bitgen}}
\newcommand{\dygen}{\myfunction{dygen}}
\newcommand{\dygentext}{\myfunctiontext{dygen}}
\newcommand{\pwcgen}{\myfunction{pwcgen}}
\newcommand{\pwcgentext}{\myfunctiontext{pwcgen}}
\newcommand{\fastgen}{\myfunction{fastgen}}
\newcommand{\fastgentext}{\myfunctiontext{fastgen}}
\newcommand{\rnd}{\myfunction{round}}

\newcommand{\sgn}{\myfunction{sgn}}
\newcommand{\dysize}{\mathfrak{L}}
\newcommand{\reach}{\myfunction{reach}}
\newcommand{\pereach}{\myfunction{pereach}}
\newcommand{\pil}{\myfunction{pil}}
\newcommand\second{{''}}
\newcommand\third{{'''}}
\newcommand\fourth{{''''}}
\newcommand{\deriv}[2]{\tfrac{\partial #1}{\partial #2}}
\newcommand\numberthis{\addtocounter{equation}{1}\tag{\theequation}}
\newcommand{\inorm}[2]{\left\lVert{#1}\right\rVert_{#2}}

\newcommand{\fnpil}[2]{%
    3*(1+tanh(2*(sin(2*pi*(#2))-1./2.)*(3+(#1))))
}
\definecolor{darkgreen}{rgb}{0.1,0.7,0.1}

\allowdisplaybreaks

\begin{document}

\title{A Universal Ordinary Differential Equation}

\author{Olivier Bournez\rsuper{a}}
\address{\lsuper{a}Ecole Polytechnique, LIX, 91128 Palaiseau Cedex, France}
\email{bournez@lix.polytechnique.fr}
\thanks{Olivier Bournez was partially supported by \emph{ANR PROJECT RACAF}}

\author{Amaury Pouly\rsuper{b}}
\address{\lsuper{b}Max Planck Institute for Software Systems (MPI-SWS), Germany}
\address{Université de Paris, IRIF, CNRS}
\email{amaury.pouly@irif.fr}

\begin{abstract}

An astonishing fact was established by Lee A. Rubel (1981): there
exists a fixed non-trivial fourth-order polynomial differential algebraic
equation (DAE) such that for any positive continuous function
$\varphi(t)$ on the reals, and for any positive continuous function
$\epsilon(t)$, it has a $\mathcal{C}^\infty$ solution with
$| y(t) - \varphi(t) | < \epsilon(t)$ for all $t$.  Lee A. Rubel
provided an explicit example of such a polynomial DAE. Other
examples of universal DAE
have later been proposed by other authors.
However, Rubel's DAE \emph{never}
has a unique solution, even with a finite number of conditions of
the form $y^{(k_i)}(a_i)=b_i$. 

The question whether one can require the solution that approximates
$\varphi(t)$ to be the unique solution for a given initial data is a well 
known open problem [Rubel 1981, page~2], [Boshernitzan 1986, Conjecture 6.2].
In this article, we solve it and show that
Rubel's statement holds for polynomial ordinary differential equations (ODEs), and since
polynomial ODEs have a unique solution given an initial data, this positively
answers Rubel's open problem. More precisely, we show that there exists a \textbf{fixed}
polynomial ODE such that for any $\varphi(t)$ and $\epsilon(t)$ there exists
some initial condition that yields a solution that is $\epsilon(t)$-close to
$\varphi(t)$ at all times.
 In particular, the solution to the ODE is necessarily
analytic, and we show that the initial condition is computable from the target function and error
function.
\end{abstract}

\maketitle

\section{Introduction}

An astonishing result was established by Lee A. Rubel in 1981 
\cite{Rub81}. There exists a universal fourth-order algebraic
differential equation in the following sense.

\begin{thmC}[\cite{Rub81}]
There exists a non-trivial fourth-order implicit differential algebraic
equation 
\begin{equation} \label{rubelgeneral}
P(y^\prime,y\second,y\third,y\fourth)=0
\end{equation}
where $P$ is a polynomial in four variables with integer coefficients,
such that for any continuous function $\varphi$ on $(-\infty,\infty)$
and for any positive continuous function $\epsilon(t)$ on
$(-\infty,\infty)$, there exists a $\mathcal{C}^\infty$ solution $y$
to \eqref{rubelgeneral}
such that $$| y(t) - \varphi(t) | < \epsilon(t)$$ for all $t \in
(-\infty,\infty)$.
\end{thmC}

Even more surprising is the fact that Rubel provided an explicit example of such
a polynomial $P$ that is particularly simple:
\begin{equation} \label{eq:rubel}
\begin{array}{lll}
3 {y ^\prime }  ^4  y ^{\second}  {y\fourth}^2 
& - 4 {y^\prime} ^4  {y \third}^2 y ^{\fourth} 
+ 6 {y^\prime}^3 {y\second}^2 y\third y\fourth 
+24 {y^\prime}^2 {y\second}^4 y\fourth & \\
&
- 12 {y^\prime}^3 y\second {y\third}^3
-29 {y^\prime}^2 {y\second}^3 {y\third}^2
+12 {y\second}^7 & = 0.
\end{array}
\end{equation}

While this result looks very surprising at first sight, Rubel's proofs
turns out to use basic arguments, and can be explained as follows. It uses the
following classical trick to build $\mathcal{C}^\infty$ piecewise
functions: let $$g(t)= \left\{
\begin{array}{ll}
e^{-1/(1-t^2)} & \mbox{ for } -1<t<1 \\
0 & \mbox{ 
otherwise.}
\end{array}\right.
$$ It is not hard to see that function $g$ is $C^\infty$ and Figure~\ref{fig:bump}
shows that $g$ looks like a ``bump''.
Since it satisfies $$\tfrac{g'(t)}{g(t)} = - \frac{2t}{(1-t^2)^2},$$
then $$g'(t) (1-t^2)^2 + g(t) 2t=0$$ and $f(t)= \int_0^t g(u) du$
satisfies the polynomial differential algebraic equation
$$f'' (1-t^2)^2 + f'(t) 2t=0.$$
Since this equation is homogeneous, it also holds for $af+b$ for any $a$ and $b$.
The idea is then to obtain a fourth order DAE that is
satisfied by every function $y(t)=\gamma f(\alpha t + \beta) + \delta$,
for all $\alpha,\beta,\gamma,\delta$. After some computations, Rubel
obtained the universal differential equation \eqref{eq:rubel}.

\begin{figure} 
\begin{center}
\begin{tikzpicture}[xscale=1.2,yscale=7]
    \draw[->] (0,-0.01) -- (0,0.4);
    \draw[->] (-2.2,0) -- (2.2,0) node[right] {$t$};
    \draw[color=red,thick] plot[id=rubel_bump_left,domain=-2:-1]
        function{0};
    \draw[color=red,thick] plot[id=rubel_bump,domain=-1:1,samples=100]
        function{exp(-1/(1-x*x))};
    \draw[color=red,thick] plot[id=rubel_bump_right,domain=1:2]
        function{0};
\end{tikzpicture}
\hspace{1cm}
\begin{tikzpicture}[xscale=3,yscale=2.2]
    \draw[->] (0,0.7) -- (0,2.1);
    \draw[->] (-0.02,1) -- (1.6,1) node[right] {$t$};
    \draw[color=red,thick] plot file{rubel_s_modules.table};
\end{tikzpicture}
\caption{On left, graphical representation of function $g$. On right,
  two $S$-modules glued together.}\label{fig:bump}
\end{center}
\end{figure}
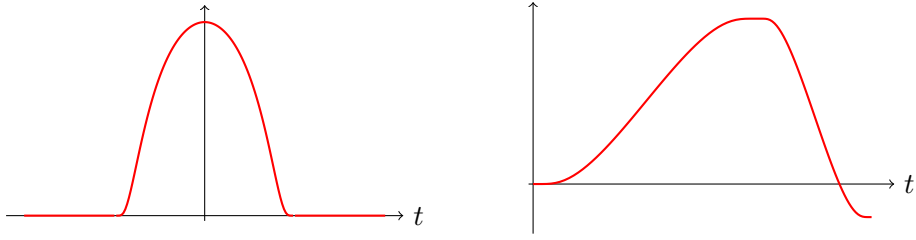

Functions of the type $y(t)=\gamma f(\alpha t + \beta) + \delta$ generate what
Rubel calls $S$-modules: a function that values $A$ at $a$, $B$
at $b$, is constant on $[a,a+\delta]$, monotone on
$[a+\delta,b-\delta]$, constant on $[b-\delta,b]$, by an appropriate
choice of $\alpha,\beta, \gamma, \delta$.
Summing $S$-modules corresponds to gluing then together, as is depicted in
Figure \ref{fig:bump}. Note that finite, as well as infinite sums\footnote{With some convergence or disjoint domain conditions.}
of $S$-modules
still satisfy the equation \eqref{eq:rubel} and thus any piecewise affine function
(and hence any continuous function) can be approximated by an appropriate sum of
$S$-modules. This concludes Rubel's proof of universality.
\bigskip

As one can see, the proof turns out to be frustrating because the equation essentially
allows any behavior. This may be interpreted as merely
stating that differential algebraic equations is simply too lose a model.
Clearly, a key point is that this differential
equation does not have a unique solution for any given initial
condition: this is the core principle used to glue a finite or infinite
number of $S$-modules and to approximate any continuous function.
Rubel was aware of this issue and left open the following question in \cite[page 2]{Rub81}.
\begin{quote}
``It is open whether we can require
in our theorem that the solution that approximates $\varphi$ to be the
unique solution for its initial data.''
\end{quote}
Similarly, the following is conjectured in \cite[Conjecture
6.2]{boshernitzan1986universal}.
\begin{quote}
``Conjecture. There exists a non-trivial differential algebraic
equation such that any real continuous function on $\R$ can be
uniformly approximated on all of $\R$ by its real-analytic solutions''
\end{quote}

The purpose of this paper is to provide a positive answer to both
questions. We prove that a fixed polynomial ordinary differential
equations (ODE) is universal in above Rubel's sense.
At a high level, our proofs are based on ordinary
differential equation programming. This programming is inspired by
constructions from our previous paper \cite{ICALP2016}. Here, we mostly use
this programming technology to achieve a very different goal and to provide
positive answers to these above open problems.

We also believe they open some lights on
computability theory for continuous-time models of computations. In
particular, it follows that concepts similar to Kolmogorov complexity
can probably be expressed naturally by measuring the complexity of
the initial data of a (universal-) polynomial ordinary differential equation for a
given function. We leave this direction for future work.

The current article is an extended version of \cite{ICALP2017}: here
all proofs are provided, and we extend the statements by proving 
that the initial condition can always be computed from the function in
the sense of Computable Analysis. 

\subsection{Related work and discussions}

First, let us mention that  Rubel's universal differential equation has been extended
in several papers. In particular, Duffin proved in \cite{Duffin81}
that implicit universal differential equations with simpler expressions exists, such as
$$n^2 y^{''''} {y'}^2 + 3 n (1-n) y^{'''} y^{''} y' + (2n^2 -3n +1)
{y^{''}}^3=0$$
for any $n >3.$
The idea of \cite{Duffin81} is basically to replace the
$\mathcal{C}^\infty$ function $g$ of \cite {Rub81} by some piecewise
polynomial of fixed degree, that is to say by splines. Duffin also
proves that considering trigonometric polynomials for function $g(x)$
leads to the universal differential equation $$n  y^{''''} {y'}^2 + (2-3n)
y^{'''} y^{''} y' + 2 (n-1) {y''}^3=0.$$
This is done at the price of approximating function
$\varphi$ respectively by splines or trigonometric splines solutions
which are $\mathcal{C}^n$ (and $n$ can be taken arbitrary big) but not
$\mathcal{C}^\infty$ as in \cite{Rub81}.
 Article
\cite{Briggs02} proposes another universal differential equation whose
construction is based on Jacobian elliptic functions. Notice that
\cite{Briggs02} is also correcting some statements of
\cite{Duffin81}.

All the results mentioned so far are concerned with approximations of continuous
functions over the whole real line. Approximating functions over a compact domain seems
to be a different (and somewhat easier for our concerns) problem, since basically
by compactness, one just needs to approximate the function locally on a finite number
of intervals. A 1986
reference survey discussing both approximation over the real line and
over compacts is \cite{boshernitzan1986universal}.
Recently, over compact domains, the existence of universal ordinary
differential equation $\mathcal{C}^\infty$ of order $3$ has been
established in \cite{couturier2016construction}: it is
shown that for any $a<b$, there exists a third order
$\mathcal{C}^\infty$ differential equation $y'''=F(y,y',y'')$ whose
solutions are dense in $\mathcal{C}^0([a,b])$.  Notice that this is
not obtained by explicitly stating such an order $3$ universal
ordinary differential, and that this is a weaker notion of
universality as solutions are only assumed to be arbitrary close over a
compact domain and not all the real line.  Order $3$ is argued to be a
lower bound for Lipschitz universal ODEs
\cite{couturier2016construction}.

Rubel's result has sometimes been considered to be the
equivalent, for analog computers, of the universal Turing machines. This
includes Rubel's paper motivation given in \cite[page 1]{Rub81}. We
now discuss and challenge this statement.

Indeed, differential algebraic equations are known to be related to the
General Purpose Analog Computer (GPAC) of Claude Shannon \cite{Sha41},
proposed as a model of the Differential Analysers \cite{Bus31}, a
mechanical programmable machine, on which he worked as an
operator. Notice that the original relations stated by Shannon in
\cite{Sha41} between differential algebraic equations and GPACs have
some flaws, that have been corrected later by \cite{Pou74} and
\cite{GC03}. Using the better defined model of GPAC of \cite{GC03}, it
can be shown that functions generated by GPAC exactly correspond
to polynomial ordinary differential equations. Some recent results
have established that this model, and hence polynomial ordinary
differential equations can be related to classical computability
\cite{JOC2007} and complexity theory \cite{ICALP2016}.

However, we do not really agree with the statement that Rubel's result is
the equivalent, for analog computers, of the universal Turing
machines. In particular, Rubel's notion of universality is completely different
from those in computability theory.
For a given initial data, a (deterministic) Turing machine has only
one possible evolution. On the other hand, Rubel's equation does not dictate any
evolution but rather some conditions that any evolution has to satisfy. In other
words, Rubel's equation can be interpreted as the equivalent of an invariant of
the dynamics of (Turing) machines, rather than a universal machine in the sense
of classical computability.

Notice that while several results have established that (polynomial)
ODEs are able to simulate the evolution of
Turing machines (see e.g. \cite{JOC2007, gcb08, ICALP2016}), the
existence of a universal ordinary differential equation does not
follow from them. To understand the difference, let us restate the main result of \cite{gcb08},
of which \cite{ICALP2016} is a more advanced version for polynomial-time computable
functions.

\begin{thm}\label{th:gcb08}
A function $f:[a,b]\to\R$ is computable (in the framework of Computable Analysis)
if and only if there exists some polynomials $p:\R^{n}\to\R^n$, $p_0:\R\to\R$
with computable coefficients and $\alpha_1,\ldots,\alpha_{n-1}$ computable reals such that
for all $x\in[a,b]$, the solution $y:\Rp\to\R^n$ to the Cauchy problem
\[y(0)=(\alpha_1,\ldots,\alpha_{n-1},p_0(x)),\qquad y'=p(y)\]
satisfies that for all $t\geqslant 0$ that
\[|f(x)-y_1(t)|\leqslant y_2(t)\quad\text{and}\quad \lim_{t\rightarrow\infty}y_2(t)=0.\]
\end{thm}

Since there exists a universal Turing machine, there exists a ``universal'' polynomial ODE
for computable functions. But there are major differences between Theorem~\ref{th:gcb08}
and the result of this paper (Theorem~\ref{th:universal_pivp}). Even if we have a strong
link between the Turing machines's configuration and the evolution of the differential
equation, this is not enough to guarantee what the trajectory of the system
will be at all times. Indeed, Theorem~\ref{th:gcb08} only guarantees that $y_1(t)\rightarrow f(x)$
asymptotically. On the other hand, Theorem~\ref{th:universal_pivp} guarantees the value
of $y_1(t)$ at all times.
Notice that our universality result also applies to functions
that are not computable (in which case the initial condition is computable from
the function but still not computable).
\bigskip

We would like to mention some implications for experimental sciences that are related to the
classical use of ODEs  in such contexts. Of
course, we know that this part is less formal from a mathematical
point of view, but we believe this discussion has some importance:
A key property in experimental sciences, in particular physics is
analyticity. Recall that a function is analytic if it is equal to its
Taylor expansion in any point. It has sometimes been observed that
``natural'' functions coming from Nature are analytic, even if this
cannot be a formal statement, but more an observation (see
e.g. \cite{CIEChapter2007,Moo90,KM99}). 
We obtain a fixed universal polynomial ODE, so in particular all its solution
must be analytic\footnote{Which is not the case for polynomial DAEs.}, and it
follows that universality holds even
with analytic functions. All previous constructions mostly worked
by gluing together $\mathcal{C}^\infty$ or $\mathcal{C}^n$ functions,
and as it is well known ``gluing'' of analytic functions is
impossible. We believe this is an important difference with previous works.

As we said, Rubel's proof can be seen as an indication that (fourth-order)
polynomial implicit DAE is too loose model compared to classical ODEs, allowing in particular
to glue solutions together to get new solutions. As observed in many articles citing Rubel's
paper, this class appears so general that from an experimental point of view, it makes
littles sense to try to fit a differential model because a single equation can
model everything with arbitrary precision. Our result implies the same for polynomial ODEs
since, for the same reason, a single equation of sufficient dimension can model
everything.

Notice that our constructions have at the end some similarities with
Voronin's theorem.  This theorem states that Riemann's $\zeta$
function is such that for any analytic function $f(z)$ that is
non-vanishing on a domain $U$ homeomorphic to a closed disk, and any
$\epsilon>0$, one can find some real value $t$ such that for all
$z \in U$, $|\zeta(z+it)-f(z)| <\epsilon$.  Notice that $\zeta$
function is a well-known function known not to be solution of any
polynomial DAE (and consequently polynomial ODE), and hence
there is no clear connection to our constructions based on ODEs. We
invite to read the post \cite{blogvoronin} in ``G\"odel's Lost
Letter and P=NP'' blog for discussions about potential implications of this
surprising result to computability theory.

\subsection{Formal statements}

Our results are the following:

\begin{thm}[Universal PIVP]\label{th:universal_pivp}
There exists a \textbf{fixed} polynomial vector $p$ in $d$ variables with rational coefficients such that for any functions
$f\in C^0(\R)$ and $\varepsilon\in C^0(\R,\Rps)$, there exists $\alpha\in\R^d$
such that there exists a unique solution $y:\R\rightarrow\R^d$ to
$y(0)=\alpha$, $y'=p(y).$
Furthermore, this solution satisfies that
$|y_1(t)-f(t)|\leqslant\varepsilon(t)$
for all $t\in\R$, and it is analytic.

Furthermore, $\alpha$ can be computed from $f$ and $\varepsilon$ in
the sense of Computable Analysis, more precisely
$(f,\varepsilon)\mapsto\alpha$ is $([\rho\to\rho]^2,\rho^d)$-computable (refer to Section \ref{sec:ccastuff}
for formal definitions). 
\end{thm}

It is well-known that polynomial ODEs
can be transformed into DAEs that have the same analytic solutions, see \cite{CPSW05}
for example. The following then follows for DAEs.

\begin{thm}[Universal DAE]\label{th:universal_dae}
There exists a \textbf{fixed} polynomial $p$ in $d+1$ variables with rational coefficients such that for any functions
$f\in C^0(\R)$ and $\varepsilon\in C^0(\R,\Rps)$, there exists $\alpha_0,\ldots,\alpha_{d-1}\in\R$
such that there exists a unique \textbf{analytic} solution $y:\R\rightarrow\R$ to
$y(0)=\alpha_0,y'(0)=\alpha_1,\ldots,y^{(d-1)}(0)=\alpha_{d-1}$, $ p(y,y',\ldots,y^d)=0.$
Furthermore, this solution satisfies that
$|y(t)-f(t)|\leqslant\varepsilon(t)$
for all $t\in\R$.

Furthermore, $\alpha$ can be computed from $f$ and $\varepsilon$ in
the sense of Computable Analysis, more precisely
$(f,\varepsilon)\mapsto\alpha$ is $([\rho\to\rho]^2,\rho^d)$-computable (refer to Section \ref{sec:ccastuff}
for formal definitions). 
\end{thm}


\begin{rem}
  Notice that both theorems apply even when $f$ is not computable. In this case,
  the initial condition(s) $\alpha$ exist but are not computable.
  We will prove that $\alpha$ is always \emph{computable from} $f$ and $\varepsilon$,
  that is the mapping $(f,\varepsilon)\mapsto\alpha$ is computable in the framework
  of Computable Analysis, with an adequate representation of $f,\varepsilon$ and $\alpha$.
\end{rem}

\begin{rem}
  Notice that we do not provide explicitly in this paper the
  considered polynomial ODE, nor its dimension $d$. But it can be
  derived by following the constructions. We currently estimate $d$ to be
  more than three hundred following the precise constructions of this
  paper (but also to be very far from the optimal). We did not try to minimize $d$ in the current
  paper, as we think our results are sufficiently hard to be followed
  in this paper for not being complicated by considerations about
  optimization of dimensions.
\end{rem}

\begin{rem}\label{rem:domain_of_def}
    Both theorems are stated for \emph{total functions} $f$ and $\varepsilon$ over $\R$. It
    trivially applies to any continuous partial function that can be extended to
    a continuous function over $\R$. In particular, it applies to any functions
    over $[a,b]$. It is not hard to see that it also applies to functions over $(a,b)$
    by rescaling $\R$ into $(a,b)$ using the cotangent:
    \[z(t)=y\left(-\cot\left(\tfrac{t-a}{b-a}\pi\right)\right)
    \quad\text{satisfies}\quad z'(t)=\phi'(t)p(z(t)),\quad\phi'(t)=\tfrac{\pi}{b-a}(1+\phi(t)^2).\]
    More complex domains such as $[a,b)$ and $(a,b]$ (with $a$ possibly infinite) can also be obtained using a similar method.
  \end{rem}

\begin{rem}
    Since the solution of a polynomial (or analytic) differential equation is analytic, our
    results can be compared with the problem of
    building uniform approximations of continuous function
    on the real line by analytic ones, and hence can be seen as a
    strengthening of such results  (see e.g. \cite{kaplan1955}). 
\end{rem}

\begin{rem}\label{rem:domain_of_def_sol}
    Let $Y(\alpha)$ be the solution given by Theorem~\ref{th:universal_pivp} satisfying $Y(\alpha)(0)=\alpha$.
    Note that the theorem does \emph{not} specify the existence of $Y(\alpha)(t)$ for all $t$ and $\alpha$. In fact, because
    of function $\fastgen$ in what follows, $Y(\alpha)$ will explode in finite time for all $\alpha$ that have certain coordinates rational,
    and the length of the interval of life depends on $\alpha$. Therefore, given $\alpha\in\R^d$, any ball around $\alpha$
    contains a $\beta$ such that $Y(\beta)$ explodes in finite time
    for the function $Y$ corresponding to our constructions.
  \end{rem}
  
\begin{rem}
    It may look at first like that Theorem~\ref{th:universal_pivp} violates  Brouwer's Invariance of domain
    but this is not the case. 
    Indeed, continuing with the notation of above remark, $Y$ is continuous\footnote{This is the local continuity
    of the solution to a smooth differential equation with respect to the initial condition.} and $Y$ is injective\footnote{This
    is the fact that for an autonomous ODE, two trajectories are either disjoint or the same.}
    with image in $S:=\bigcup_{a<b}C^0((a,b),\R^d)$ (see Remark~\ref{rem:domain_of_def_sol} about domains).
    Clearly $S$ is of much higher dimension than $\alpha\in\R^d$ but $Y$ is not dense in $S$ so there is no contradiction.
    On the other hand, if we only consider the first coordinate $Y_1$, then $Y_1$ is dense in $C^0(\R,\R)$
    but is not injective.
\end{rem}

\subsection{Overview of the proof}

A first a priori difficulty is that if one considers a fixed
polynomial ODE $y'=p(y)$, one could think that the growth
of its solutions is constrained by $p$ and thus cannot be arbitrary. This would
then prevent us from building a universal ODE simply because it could not grow
fast enough.
This fact is related to Emil Borel's conjecture in
\cite{BorelConjecture} (see also \cite{hardy1912some}) that a solution, \emph{defined over $\R$},
to a system with $n$ variables has growth bounded by roughly $e_n(x)$, the $n-$th iterate of $\exp$.
The conjecture is proved for $n=1$
\cite{BorelConjecture}, but has been proven to be false for $n=2$ in \cite{vijayaraghavan1932croissance} and \cite{BBV37}.
Bank \cite{Bank75} then adapted the previous counter-examples to provide a
DAE whose non-unique increasing real-analytic solutions at infinity do not have
any majorant. See the discussions (and Conjecture 6.1) in
\cite{boshernitzan1986universal} for discussions about the growth of
solutions of DAEs, and their relations to functions $e_n(x)$.

Thus, the first important part of this paper is to refine Bank's counter-example
to build $\fastgentext$, a fast-growing function that satisfies even stronger properties.
The second major ingredient is to be able to approximate a function with arbitrary precision everywhere.
Since this is a difficult task, we use $\fastgentext$ to our advantage to show that
it is enough to approximate functions that are bounded and change slowly (think 1-Lipschitz, although the exact condition
is more involved). That is to say, to deal with the case where there
is no problem about the growth and rate of change of functions in some way.
This is the purpose of the function $\pwcgentext$ which can build arbitrary
almost piecewise constant functions as long as they are bounded and change slowly.

It should be noted that in the entire paper, we construct \emph{generable
functions} (in several variables) (see Section~\ref{sec:generable}).
For most of the constructions, we only use basic facts
like the fact that generable functions are stable under arithmetic, composition
and ODE solving. We know that generable functions satisfy polynomial
partial equations and use this fact only at the very end to show that the generable
approximation that we have built, in fact, translates to a polynomial
ordinary differential
equation.

The rest of the paper is organised as follows. In Section
\ref{sec:stuff}, we recall some concepts and results from other
articles. The main purpose of this section is to present Theorem
\ref{th:pereach}. This theorem is the analog equivalent of doing an assignment
in a periodic manner. Section \ref{sec:fast} is
devoted to $\fastgentext$, the fast-growing function. In Section
\ref{sec:dyadic}, we show how to generate a sequence of dyadic
rationals. In Section \ref{sec:bits}, we show how to generate a
sequence of bits. In Section \ref{sec:almost}, we show how to leverage the two
previous sections to generate arbitrary almost piecewise constant functions. Section \ref{sec:mainth}
is then devoted to the proof of our main theorem.

\section{Concepts and results from previous work} 
\label{sec:stuff}

\subsection{Generable functions}\label{sec:generable}

The following concept can be attributed to \cite{Sha41}: a function $f:\R\to\R$ is said to be a
PIVP (Polynomial Initial Value Problem) function if there
exists a system of the form $y'=p(y)$, where $p$ is a (vector of) polynomial, with $f(t)=y_1(t)$ for all $t$, where $y_1$
denotes first component of the vector $y$ defined in $\R^d$. 
We need in our proof to extend this concept to talk about
multivariate functions. In \cite{BournezGP16a}, we introduced the following class, which can
be seen as extensions of \cite{GBC09}. Let $\K$ be the smallest generable field
(see \cite{BournezGP16a} for formal definitions and properties), the reader only needs to know that $\Q\subseteq\K\subseteq\R_P$ where
$\R_P$ is the set of polynomial-time computable reals, and $\K$ is closed under
images of generable functions.

\begin{defi}[Generable function]\label{def:gpac_generable_ext}
Let $d,e\in\N$, $I$ be an open and connected subset of $\R^d$
and $f:I\rightarrow\R^e$. We say that $f$ is generable if and only if
there exists an integer
$n\geqslant e$, a $n\times d$ matrix $p$ consisting of polynomials with coefficients in $\K$
, $x_0\in\K^d$, $y_0\in\K^n$
and $y:I\rightarrow\R^n$ satisfying for all $x\in I$:
\begin{itemize}
\item $y(x_0)=y_0$ and $J_y(x)=p(y(x))$ 
  \hfill$\blacktriangleright$ $y$ satisfies a polynomial differential
  equation\footnote{$J_y$ denotes the Jacobian matrix of $y$.},
\item $f(x)=(y_1(x), \ldots, y_e(x))$\hfill$\blacktriangleright$ the components of $f$ are components of $y$.
\end{itemize}
\end{defi}

This class strictly generalizes functions generated by polynomial ODEs. Indeed,
in the special case of $d=1$ (the domain of the function has dimension $1$),
the above definition is equivalent to saying that $y'=p(y)$ for some polynomial $p$.
The interested reader can read more about this in \cite{BournezGP16a}.

For the purpose of this paper, we will need to consider a slight generalisation of
this notion where the initial condition is considered to be (depending
of) a parameter, therefore
defining not just a single function but a family of function, and most importantly,
\emph{all sharing the same differential equation}. Formally:

\begin{defi}[Uniformly-generable function]\label{def:gpac_unif_generable_ext}
    Let $d,m,e\in\N$, $I$ be an open and connected subset of $\R^d$, $\Gamma\subseteq\R^m$,
    and $f:\Gamma\times I\rightarrow\R^e$. We say that $f$ is uniformly-generable if and only if
    there exists an integer $n\geqslant e$, a $n\times d$ matrix $p$ consisting of polynomials
    with coefficients in $\K$, $x_0\in\K^d\cap I$ and a $(\rho^m,\rho^n)-$computable
    function $y_0:\Gamma\to\R^n$ such that for all $\gamma\in\Gamma$, there
    exists $y:I\rightarrow\R^n$ satisfying for all $x\in I$:
    \begin{itemize}
    \item $y(x_0)=y_0(\gamma)$ and $J_y(x)=p(y(x))$ 
    \hfill$\blacktriangleright$ $y$ satisfies a polynomial differential equation
    \item $f(\gamma;x)=(y_1(x), \ldots, y_e(x))$\hfill$\blacktriangleright$ the components of $f$ are components of $y$.
    \end{itemize}
\end{defi}

For readability, we will distinguish parameters from variables using a semicolon,
for example $f(\gamma;x)$ is parameterized by $\gamma$. This should make it clear
from the context what is considered as parameter and what is
considered as a variable.

\begin{rem}\label{rem:unif_gen_coeff}
    Although we have chosen $x_0$ and the coefficients of $p$ to be in $\K$ in the
    definition above, it is clear that we can change this set \emph{at the cost of
    increasing the set $\Gamma$ of parameters}. For example we could take all coefficients
    to be rational or in $\set{0,1}$ by adding one extra parameter per coefficient
    and hence ``hiding'' them in $y_0$. The only real constraint is that since $y_0$ must
    remain computable, we still need all elements of $\K$ to be computable.
\end{rem}

For the purpose of this paper, the reader only needs to know that the class of
generable functions enjoys many stability properties that make it easy to create new
functions from basic operations. Informally, one can add, subtract, multiply, divide
and compose them at will, the only requirement is that the domain of definition
must always be connected. In particular, the class of generable functions contains
some common mathematical functions:
\begin{itemize}
\item (multivariate) polynomials;
\item trigonometric functions: $\sin$, $\cos$, $\tan$, etc;
\item exponential and logarithm: $\exp$, $\ln$;
\item hyperbolic trigonometric functions: $\sinh$, $\cosh$, $\tanh$.
\end{itemize}
Two famous examples of functions that are \emph{not} in this class are the $\zeta$
and $\Gamma$, we refer the reader to \cite{BournezGP16a} and \cite{GBC09}
for more information.

A nontrivial fact is that generable functions are
always analytic. This property is well-known in the one-dimensional case but is
less obvious in higher dimensions, see \cite{BournezGP16a} for more details.
Moreover, generable functions satisfy the following crucial properties.

\begin{lem}[Closure properties of generable functions \cite{BournezGP16a}] \label{lemma:closure} 
Let $f:\subseteq\R^d\rightarrow\R^n$ and
$g:\subseteq\R^e\rightarrow\R^m$ be generable functions. Then $f+g$, $f-g$, $fg$, $\tfrac{f}{g}$ and
$f\circ g$ are generable\footnote{With the obvious dimensional condition associated with
each operation.}.
\end{lem}

\begin{lem}[Generable functions are closed under ODE \cite{BournezGP16a}]\label{lem:gpac_ext_ivp_stable}
Let $d\in\N$, $J\subseteq\R$ an interval,
$f:\subseteq\R^d\rightarrow\R^d$ generable, $t_0\in J\cap\K$ and $y_0\in\dom{f}\cap\K^d$.
Assume there exists $y:J\rightarrow\dom{f}$ satisfying
\[y(t_0)=y_0,\qquad y'(t)=f(y(t))\]
for all $t\in J$, then $y$ is generable (and unique).
\end{lem}

Those results can be generalised to uniformly-generable functions with the obvious restrictions
on the domains and the roles of parameters. For example, if $f(\alpha;x)$ and $g(\beta;y)$
are uniformly-generable over $A\times X$ and $B\times Y$ respectively, then
$h(\alpha,\beta;y):=f(\alpha;g(\beta;y))$ is uniformly-generable over $A\times B\times Y$.
We will use those facts implicitly, and in particular the following result:

\begin{thm}[Uniformly-generable functions are closed under ODE]\label{th:unif_gpac_ext_ivp_stable}
    Let $d,m\in\N$, $\Gamma\subseteq\R^m$,
    $t_0\in\K$, $J$ an open interval containing $t_0$, $f_0:\Gamma\to\R^d$ a $(\rho^m,\rho^d)-$computable
    function and $F:\subseteq\Gamma\times\R^d\rightarrow\R^d$ uniformly-generable.
    Assume that there exists $f:\Gamma\times J\rightarrow\R^d$ satisfying\footnote{We are assuming
    that for all $\gamma\in\Gamma$, $(\gamma;f(\gamma;t))\in\dom{F}$.}
    \[f(\gamma;t_0)=f_0(\gamma),\qquad \deriv{f}{t}(\gamma;t)=F(\gamma;f(\gamma;t))\]
    for all $\gamma\in\Gamma$ and $t\in J$.
    Then $f$ is uniformly-generable (and unique).
\end{thm}

\begin{proof}
    Apply Definition~\ref{def:gpac_unif_generable_ext} to $F$ to get $n\in\N$, $x_0\in\K^n$,
    $y_0:\Gamma\to\R^n$ computable and $p$ polynomial matrix with coefficients in $\K$.
    Then given $\gamma\in\Gamma$, there exists $y:\subseteq \R^d\to\R^n$ such that
    \[y(x_0)=y_0(\gamma),\qquad J_y(x)=p(y(x))\]
    and $F(\gamma;x)=(y_1(x),\ldots,y_d(x))=:y_{1..d}(x)$ for all $(\gamma,x)\in\dom{F}$.
    Let $z(t)=y(f(\gamma;t))$ which is well-defined by assumption and
    check that
    \[
        z'(t)=J_y(f(\gamma;t))\deriv{f}{t}(\gamma;t)
            =p(y(f(\gamma;t)))F(\gamma;f(\gamma;t))
            =p(z(t))z_{1..d}(t)=q(z(t))
    \]
    for some polynomial $q$ that does not depend on $\gamma$, and $z(t_0)=y(f(\gamma;t_0))=y(f_0(\gamma))$
    which is a computable function of $\gamma$ since $f_0$ is computable and $(\gamma,x)\mapsto y(x)$
    is also computable (note that $y$ depends on $\gamma$) by Proposition~\ref{prop:gen_implies_computable}.
\end{proof}

An important point, which we have in fact already used in the proof of the previous proposition,
is that generable functions are always computable, in the sense of Computable Analysis.

\begin{prop}[Generable implies computable]\label{prop:gen_implies_computable}
    Assume $f:\Gamma\times I\rightarrow\R^e$ is uniformly generable
    according to Definition \ref{def:gpac_unif_generable_ext}:
    Hence there is a $(\rho^m,\rho^n)-$computable
    function $y_0:\Gamma\to\R^n$ and a  $n\times d$ matrix $p$ consisting of polynomials
    with coefficients in $\K$, $x_0\in\K^d$ that define 
    $y(\gamma;x_0)=y_0(\gamma)$ and $J_y(x)=p(y(\gamma;x))$.
    Then the function that maps $(\gamma,x)\in\Gamma\times I$ to $y(\gamma;x)$ 
    is $([\rho^m,\rho^d],\rho^n)-$computable. 
\end{prop}

\begin{proof}
We established in proposition \cite[Proposition 31]{BournezGP16a} that
$y(x)$ is necessarily real-analytic on some neighbourhood $V=V(x)$ of $x$ for
all $x$ that corresponds to some point of the domain of $f$.

Some explicit upper bound on the radius of convergence is provided by \cite[Theorem
5]{pouly2016computational}: Assuming $t_0=0$, $k=\deg(p) \ge 2$,
$\alpha=\max(1,\inorm{y_0}{})$, the radius is at least $1/M$ with $M=M(y_0)=
(k-1) \Sigma p \alpha^{k-1}$, where $\Sigma p$ is basically the sum of the
absolute value of the coefficients of polynomials in matrix $p$.

Consequently, using classical techniques for evaluating a converging
power series whose convergence radius is known up to a given precision
(by restricting the sum up to suitable index) we get that $y$ is
computable over the ball $V(y_0)$ of radius $1/(2M)$.

Computability of $y$ then follows from classical analytic
continuation techniques: A Turing machine can then extend the
computation starting from a new point $y_1$ in $V(y_0(\gamma))$, and then
repeat the above process to compute $y$ over some ball $V(y_1)$ of
radius $1/(2M(y_1))$, and so on. Repeating the process, eventually, it will
reach $x$ and will be able to compute $y(x)$.  Refer to 
\cite{kawamura_et_al:LIPIcs:2018:9612,HolgerThiesPhD} for similar techniques and a
finer complexity analysis. 
%
\end{proof}

\subsection{Helper functions and constructions}\label{sec:helper_gen}

We mentioned earlier that a number of common mathematical functions are generable.
However, for our purpose, we will need less common functions that one can consider
to be programming gadgets.

\begin{rem}
In this subsection, some of the functions will be introduced as mapping arguments to
value, i.e. as usual mathematical functions, but some others by the properties of their
solutions (e.g. $\reach$, $\pereach$, $\pil)$. In the latter case, an explicit expression
of a function satisfying those properties can be found in the proof.
\end{rem}

One such operation is rounding (computing the nearest integer).
Note that, by construction, generable functions are analytic and in particular
must be continuous. It is thus clear that we cannot build a perfect rounding function
and in particular we have to compromise on two aspects:
\begin{itemize}
\item we cannot round numbers arbitrarily close to $n+\tfrac{1}{2}$ for $n\in\Z$ because of continuity:
    thus the function takes a parameter $\lambda$ to control the size of the ``zone'' around
    $n+\tfrac{1}{2}$ where the function does not round properly;
\item we cannot round without error due to the uniqueness of analytic functions:
    thus the function takes a parameters $\mu$ that controls how good the approximation must be.
\end{itemize}

\begin{lem}[Round, \cite{BournezGP16a}]\label{lem:rnd}
There exists a generable function $\rnd$ such that
for any $n\in\Z$, $x\in\R$, $\lambda>2$ and $\mu\geqslant0$:
\begin{itemize}
\item if $x\in\left[n-\frac{1}{2},n+\frac{1}{2}\right]$
    then $|\rnd(x,\mu,\lambda)-n|\leqslant\frac{1}{2}$;
\item if $x\in\left[n-\frac{1}{2}+\frac{1}{\lambda},n+\frac{1}{2}-\frac{1}{\lambda}\right]$
    then $|\rnd(x,\mu,\lambda)-n|\leqslant e^{-\mu}$.
\end{itemize}
\end{lem}

Another very useful operation is the analog equivalent of a discrete
assignment, done in a periodic manner. More precisely, we consider a
particular class of ODEs 
\[y'(t)=\pereach(t,\phi(t),y(t),g(t))\]
adapted from the constructions of \cite{ICALP2016}, where $g$ and $\phi$
are sufficiently nice functions.
Solutions to this equation alternate between two behaviours, for all $n\in\N$:
\begin{itemize}
\item During $J_n=[n,n+\tfrac{1}{2}]$, the system performs $y(t)\rightarrow\overline{g}$ for
    some $\bar{g}$ satisfying
    $\min_{t\in J_n}g(t)\leqslant\overline{g}\leqslant \max_{t\in J_n}g(t)$
    (note that this is voluntarily underspecified). So in particular,
    if $g(t)\approx\bar{g}$ over this time interval, then
    $y(t)\rightarrow\bar{g}$ and the system performs an ``assignment'' in
    the sense that $y(n+\tfrac{1}{2}):=\bar{g}$. Then $\phi$ controls how good the convergence is:
    the error is of the order of $e^{-\phi}$.
\item During $J_n'=[n+\tfrac{1}{2},n+1]$, the systems tries to keep $y$ constant,
ie $y'\approx 0$. More precisely, the system enforces that $|y'(t)|\leqslant e^{-\phi(t)}$.
\end{itemize}
As a result of this behavior, if $g(t)\approx\bar{g}$ for $t\in[n,n+\tfrac{1}{2}]$
then the system performs the ``assignment'' $y(n+1):=\bar{g}$ with some error that
is exponential small in $\phi$.

We now go to the proof of the existence of such a function $\pereach$
(formally stated as Theorem \ref{th:pereach}): 
We will need the following bound on $\tanh$, which essentially tells us that $\tanh(t)$
gets exponentially close (in $|t|$) to $\pm 1$ as $t\to\pm\infty$.

\begin{lem}\label{lem:tanh}
For any $t\in\R$, $|\tanh(t)-\sgn(t)|\leqslant e^{-|t|}$.
\end{lem}


\begin{lem}[Reach, \cite{BournezGP16b}]\label{lem:reach}
There exists a generable function $\reach$ such that for any $\phi\in C^0(\Rp)$,
$g\in C^0(\R)$ and $y_0\in\R$, the unique solution to
\[y(0)=y_0,\qquad y'(t)=\phi(t)\reach(g(t)-y(t))\]
exists over $\Rp$. Furthermore, for any $I=[a,b]\subseteq[0,+\infty)$,
if there exists $\bar{g}\in\R$ and $\eta\in\Rp$ such that $|g(t)-\bar{g}|\leqslant\eta$
for all $t\in I$, then for all $t\in I$,
\[|y(t)-\bar{g}|\leqslant\eta+\exp\left(-\int_a^t\phi(u)du\right)\qquad\text{whenever }\int_a^t\phi(u)du\geqslant1.\]
Furthermore, for all $t\in I$,
\[\min(\bar{g}-\eta,y(a))\leqslant y(t)\leqslant\max(\bar{g}+\eta,y(a))\]
and in particular
\[|y(t)-\bar{g}|\leqslant\max(\eta,|y(a)-\bar{g}|).\]
\end{lem}
\begin{proof}[Proof Remark]
The statement of \cite[Lemma~40]{BournezGP16b} only contains the first
and third inequalities but in fact the proof also contains the second inequality
(which is strictly stronger than the third but less immediate to use).
\end{proof}

\begin{figure}
    \begin{center}
    \begin{tikzpicture}[xscale=1.3,yscale=0.4,domain=0:3,samples=200]
        \draw[->] (-0.1,0) -- (3.1,0) node[right] {$t$};
        \draw[->] (1,-0.1) -- (1,7.1);
        \draw[red] plot function{\fnpil{0.1}{x}};
        \draw[blue] plot function{\fnpil{10}{x}};
    \end{tikzpicture}
    \end{center}
    \caption{Illustration of $\pil(\mu,\cdot)$ from Lemma~\ref{lem:plil} for various values of $\mu$: it has period $1$,
        is very small ($\leqslant e^{-\mu}$) half of the time, and the integral of the remaining half is at least $1$.\label{fig:pil}}
\end{figure}
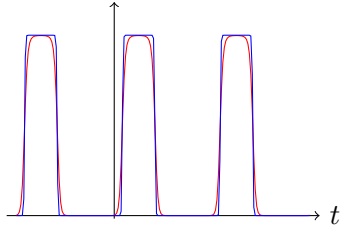

\begin{lem}[Periodic integral-low, see Figure~\ref{fig:pil}]\label{lem:plil}
There exists a generable function $\pil:\Rp\times\R\rightarrow\Rp$ such that:
\begin{itemize}
\item $\pil(\mu,\cdot)$ is 1-periodic, for any $\mu\in\Rp$;
\item $\int_0^{1/2}\pil(\mu(t),t)dt\geqslant1$ for any $\mu\in C^0(\Rp)$;
\item $|\pil(\mu,t)|\leqslant e^{-\mu}$ for any $\mu\in\Rp$ and $t\in\left[\tfrac{1}{2},1\right]$.
\end{itemize}
\end{lem}

\begin{proof}
For any $t\in\R$ and $\mu\in\Rp$, let
\[\pil(\mu,t)=A\left(1+\tanh\big(2(\sin(2\pi t)-\tfrac{1}{2})(A+\mu)\big)\right).\]
where $A=3$.
Clearly $\pil$ is generable and 1-periodic in $t$. Let $\mu\in\Rp$ and $t\in[\tfrac{1}{2},1]$,
then
\begin{align*}
    \sin(2\pi t)&\leqslant0\\
    2(\sin(2\pi t)-\tfrac{1}{2})(A+\mu)&\leqslant-A-\mu\\
    |\tanh(2(\sin(2\pi t)-\tfrac{1}{2})(A+\mu))-(-1)|&\leqslant e^{-A-\mu}&&\text{using Lemma~\ref{lem:tanh}}\\
    A|\tanh(2(\sin(2\pi t)-\tfrac{1}{2})(A+\mu))-(-1)|&\leqslant Ae^{-A-\mu}\\
    |\pil(\mu,t)|&\leqslant e^{-\mu}.\\
\end{align*}
Let $\mu\in C^0(\Rp)$ and $t\in[0,\tfrac{1}{2}]$. Observe that $\pil(\mu(t),t)\geqslant0$
and, furthermore, if $t\in[\tfrac{1}{8},\tfrac{3}{8}]$ then
\begin{align*}
    \sin(2\pi t)&\geqslant\frac{\sqrt{2}}{2}\\
    2(\sin(2\pi t)-\tfrac{1}{2})(A+\mu(t))&\geqslant\sqrt{2}-1&&\text{since }A+\mu(t)\geqslant1\\
    \tanh(2(\sin(2\pi t)-\tfrac{1}{2})(A+\mu(t)))&\geqslant\tanh(\sqrt{2}-1)\geqslant\frac{1}{3}\\
    \pil(\mu(t),t)&\geqslant \tfrac{4}{3}A.
\end{align*}
It follows that
\[
\int_0^{\tfrac{1}{2}}\pil(\mu(t),t)dt
    \geqslant\int_{\tfrac{1}{8}}^{\tfrac{3}{8}}\pil(\mu(t),t)dt
    \geqslant \left(\frac{3}{8}-\frac{1}{8}\right)\frac{4}{3}A\geqslant\frac{A}{3}\geqslant1.
    \tag*{\qedhere}
\]
\end{proof}

\begin{thm}[Periodic reach]\label{th:pereach}
There exists a generable function $\pereach:\Rp^2\times\R^2\to\R$ such that for any $I=[n,n+1]$ with $n\in\N$,
$y_0\in\R$, $\phi,\psi\in C^0(I,\Rp)$ and $g\in C^0(I,\R)$, the unique solution to
\[y(n)=y_0,\qquad y'(t)=\psi(t)\pereach(t,\phi(t),y(t),g(t))\]
exists over $I$. Furthermore,
\begin{enumerate}[label={\bfseries(\roman*)}]
\item\label{th:pereach:conv} For all $\bar{g}\in\R$ and $\eta,\theta\in\Rp$
such that $|g(t)-\bar{g}|\leqslant\eta$ and $\psi(t)\phi(t)\geqslant\theta\geqslant1$
for all $t\in[n,n+\tfrac{1}{2}]$, we have that
$|y(n+\tfrac{1}{2})-\bar{g}|\leqslant\eta+\exp\left(-\theta\right)$.
\item\label{th:pereach:bounded} For all $t\in[n,n+1]$, $\bar{g}\in\R$ and $\eta\in\Rp$
such that $|g(u)-\bar{g}|\leqslant\eta$ for all $u\in[n,t]$, we have that
$|y(t)-\bar{g}|\leqslant\max(\eta,|y(n)-\bar{g}|)$.
\item\label{th:pereach:constant} For all $t\in[n+\tfrac{1}{2},n+1]$,
$|y(t)-y(n+\tfrac{1}{2})|\leqslant\int_{n+\tfrac{1}{2}}^t\psi(u)\exp\left(-\phi(u)\right)du$.
\item\label{th:pereach:lower_conv} For all $\theta\in\Rp$ such that $\psi(t)\phi(t)\geqslant\theta\geqslant1$
for all $t\in[n,n+\tfrac{1}{2}]$,  we have that
$y(n+\tfrac{1}{2})\geqslant\min_{u\in[n,n+\tfrac{1}{2}]}g(u)-\exp\left(-\theta\right)$.
\item\label{th:pereach:lower_bound} For all $t\in [n,n+1]$,
    $\min\left(y(n),\min_{u\in[n,t]}g(t)\right)\leqslant y(t)\leqslant\max\left(y(n),\max_{u\in[n,t]}g(t)\right)$.
\end{enumerate}
\end{thm}

\begin{proof}[Proof]
Define $\pereach(t,\phi,y,g)=\pil(\phi+r^2,t)r$ where
$r=\phi\reach(g-y)$ where $\pil$ is defined in Lemma~\ref{lem:plil} and
$\reach$ is defined in Lemma~\ref{lem:reach}. Fix $n\in\N$ and $I=[n,n+1]$.
First notice that $\pil$ is nonnegative by Lemma~\ref{lem:plil} thus by Lemma~\ref{lem:reach},
the solution must exists over $I$. We now prove each point separately:
\begin{enumerate}[label={\bfseries(\roman*)}]
\item We have that
    \[\int_n^{n+\tfrac{1}{2}}\pil(\phi(u)+r(u)^2,u)\psi(u)\phi(u)du
    \geqslant\theta\int_n^{n+\frac{1}{2}}\pil(\phi(u)+r(u)^2,u)du
    \geqslant\theta\geqslant1\]
    by Lemma~\ref{lem:plil}. Thus
    $|y(n+\tfrac{1}{2})-\bar{g}|\leqslant\eta+e^{-\theta}$ by Lemma~\ref{lem:reach}.
\item Apply Lemma~\ref{lem:reach} to the interval $[n,t]$.
\item We have that
    \[|y'(t)|=|\psi(t)\pil(\phi(t)+r(t)^2,u)r(t)|\leqslant \psi(t)e^{-\phi(t)-r(t)^2}|r(t)|\leqslant \psi(t)e^{-\phi(t)}\]
    by Lemma~\ref{lem:plil}, for all $t\in[n+\tfrac{1}{2},n+1]$. The inequality follows by integration.
\item Let $m=\min_{u\in[n,n+\tfrac{1}{2}]}g(u)$ and $M=\max_{u\in[n,n+\tfrac{1}{2}]}g(u)$, define
    $\bar{g}=\tfrac{m+M}{2}$ and $\eta=\frac{M-m}{2}$. Then the assumptions of item \ref{th:pereach:conv} are satisfied
    and we get that $|y(n+\tfrac{1}{2})-\bar{g}|\leqslant\eta+e^{-\theta}$ so in particular
    $y(t)\geqslant\bar{g}-\eta-e^{-\theta}$ but $\bar{g}-\eta=m$ so this concludes.
\item The last item is more subtle because we want to use item \ref{th:pereach:bounded} but we do not know
    if $y(n)-\bar{g}$ and $y(t)-\bar{g}$ have the same sign. Let
    $m=\min_{u\in[n,t]}g(u)$ and $M=\max_{u\in[n,t]}g(u)$, define
    $\bar{g}=\tfrac{m+M}{2}$ and $\eta=\frac{M-m}{2}$.
    Then the assumptions of Lemma~\ref{lem:reach} are satisfied over $[n,t]$
    and we get that $\min(\bar{g}-\eta,y(n))\leqslant y(t)$ but $\bar{g}-\eta=m$
    so this concludes.
    \qedhere
\end{enumerate}
\end{proof}

\subsection{Computable Analysis and Representations}
\label{sec:ccastuff}
In order to prove the computability of the map
$(f,\varepsilon)\mapsto\alpha$ in Theorems~\ref{th:universal_pivp}
and~\ref{th:universal_dae}, we need to express the related notion of
computability for real numbers, functions and operators. We recall
here the related concepts:
Computable Analysis, specifically Type-2 Theory
of Effectivity (TTE) \cite{Wei00}, is a theory to study
algorithmic aspects of real numbers, functions and higher-order operators over real numbers.
Subsets of real numbers are also of great interest to this theory but will not need them in this
paper. This theory is based on classical notions of computability (and complexity) of Turing
machines which are applied to problems involving real numbers, usually by means of (effective)
approximation schemes. We refer the reader to \cite{Wei00,BHW09,ccatutorial} for  tutorials on Computable
Analysis. In order to avoid a lengthy introduction on the subject, we simply introduce the
elements required for the paper at a very high level. In what follows, $\Sigma$ is a finite alphabet.

The core concept of TTE is that of \emph{representation}: a representation of a space $X$
is simply a surjective function $\delta:\subseteq\Sigma^\omega\to X$. If $x\in X$ and $p\in\Sigma^\omega$
is such that $\delta(p)=x$ then $p$ is called a \emph{$\delta$-name} of $x$: $p$ is one way of
describing $x$ with a (potentially infinite) string. In TTE, all computations are done on infinite
string (names) using Type 2 machines, which are Turing machines operating on infinite strings but
where each bit of the output only depends on a finite prefix of the input. Type 2 machines give
rise to the notion of computable functions from $\Sigma^\omega$ to $\Sigma^\omega$. Given two
representations $\delta_X,\delta_Y$ of some spaces $X$ and $Y$, one can define two interesting notions:
\begin{itemize}
\item $\delta_X$-computable elements of $X$: those are the elements $x$ such that $\delta_X(p)=x$
    for some \emph{computable name} $p$ ($p:\N\to\Sigma$ is computable by a usual Turing machine);
\item $(\delta_X,\delta_Y)$-computable functions from $X$ to $Y$: those are the functions
    $f:\subseteq X\to Y$ for which we can find a \emph{computable realiser}
    $F$ ($F:\subseteq\sigma^\omega\to\omega$ is computable by a Type 2 machine) such that
    $f\circ\delta_X=\delta_Y\circ F$.
    \begin{center}
    \begin{tikzpicture}
        \matrix (M) [matrix of math nodes, row sep=1cm, column sep=1cm]
        {
            X & Y\\
            \Sigma^\omega & \Sigma^\omega \\
        };
        \draw[->] (M-1-1) -- (M-1-2) node[midway,above] {$f$};
        \draw[->] (M-2-1) -- (M-2-2) node[midway,above] {$F$};
        \draw[->] (M-2-1) -- (M-1-1) node[midway,left] {$\delta_X$};
        \draw[->] (M-2-2) -- (M-1-2) node[midway,right] {$\delta_Y$};
    \end{tikzpicture}
    \end{center}
\end{itemize}

In this paper, we will only need a few representations to manipule real numbers, sequences and
continuous real functions:
\begin{itemize}
\item $\nu_\N:\subseteq\Sigma^\omega\to\N$ is a representation of the
  integers. The
    details of the encoding at not very important, since natural
    representations such as 
    unary and binary representations are equivalent.
\item $\nu_\Q:\subseteq\Sigma^\omega\to\N$ is a representation of the rational numbers, again the
    details of the encoding at not very important for natural
    representations. 
\item $\rho:\subseteq\Sigma^\omega\to\R$ is the \emph{Cauchy representation} of real numbers which
    intuitively encodes a real number $x$ by a converging sequence of intervals $[l_n,r_n]\ni x$
    of rationals numbers. Alternatively, one can also use Cauchy sequences with a known rate of
    convergence.
\item $[\delta_X,\delta_Y]:\subseteq\Sigma^\omega\to X\times Y$ is the representation of
    pairs of elements of $(X,Y)$ where the first (resp. second) component uses $\delta_X$ (resp. $\delta_Y$).
    In particular, $\delta^k$ is a shorthand notation of the representation $[\delta,[\delta,[\ldots]]]$
    of $X^k$. In this paper we will often use $\rho^k$ to represent $\R^k$.
\item $\delta^\omega:\subseteq\Sigma^\omega\to X^\N$ is the representation of sequences of
    elements of $X$, represented by $\delta$. For example $\rho^\omega$ can be used to
    represent sequences of real numbers.
\item $[\delta_X\to\delta_Y]_Z: \subseteq\Sigma^\omega\to C^0(Z,Y)$ is the representation of
    \emph{continuous\footnote{Without giving too much details, this requires $X$ and $Y$ to be
    $T_0$ spaces with countable basis and $\delta_X,\delta_Y$ to be admissible. It will be enough
    to know that $\rho$ is admissible for the usual topology on $\R$.} functions} from $Z\subseteq X$ to $Y$,
    we omit $Z$ if $Z=X$.
    We will mostly need $[\rho^k\to\rho]$ which represents\footnote{Technically, is equivalent
    to a representation of.} $C^0(\R^k,\R)$ as a list of boxes which enclose the graph of the
    function with arbitrary precision. Informally, it means we can ``zoom'' on the graph of the function
    and plot it with arbitrary precision.
\end{itemize}

It will be enough for the reader to know that those representations are well-behaved. In particular,
the following functions are computable (we always use $\rho$ to represent $\R$):
\begin{itemize}
\item the arithmetical operations $+,-,\cdot,/:\subseteq\R\times\R\to\R$,
\item polynomials $p:\R^n\to\R$ with computable coefficients,
\item elementary functions $\cos,\sin,\exp$ over $\R$.
\end{itemize}
Furthermore, the following operators on continuous functions are computable:
\begin{itemize}
\item the arithmetical operators $+,-,\cdot,/:\subseteq C^0(\R)\times C^0(\R)\to C^0(\R)$,
\item composition $\circ:C^0(X,Y)\times C^0(Y,Z)\to C^0(X,Z)$,
\item inverse $\cdot^{-1}:C^0(X,Y)\to C^0(Y,X)$ for increasing (or decreasing) functions,
\item evaluation $C^0(X,Y)\times X\to Y$, $(f,x)\mapsto f(x)$.
\end{itemize}
We will also use the fact that the map $X^\N\times\N\to X, (x,i)\mapsto x_i$ is
$([\delta^\omega,\nu_N],\delta)$-computable for any space $X$ represented by $\delta$.

Refer to \cite{Wei00,BHW09,ccatutorial} for more complete
discussions, and in particular to
\cite{kawamura_et_al:LIPIcs:2018:9612,HolgerThiesPhD}  for
computability and complexity issues related to ordinary differential
equations solving. 

\section{Generating fast growing functions}\label{sec:fast}

Our construction crucially relies on our ability to build functions of arbitrary
growth. At the end of this section, we obtain a function $\fastgentext$ with a straightforward
specification: for any infinite sequence $a_0,a_1,\ldots$ of positive numbers, we can find
a suitable $\alpha\in\R$ such that $\fastgen(\alpha;n)\geqslant a_n$ for all $n\in\N$.
Furthermore, we can ensure that $\fastgen(\alpha;\cdot)$ is increasing. Notice,
and this is the key point, that the definition of $\fastgentext$ is independent
of the sequence $a$: a single generable function (and thus differential system)
can have arbitrary growth by simply tweaking its initial value.

Our construction builds on the following lemma proved by \cite{Bank75}, based on an example
of \cite{BBV37}. The proof essentially relies
on the function $\tfrac{1}{2-\cos(x)-\cos(\alpha x)}$ which is generable and well-defined
for all positive $x$ if $\alpha$ is irrational. By carefully choosing $\alpha$,
we can make $\cos(x)$ and $\cos(\alpha x)$ simultaneously arbitrary close to $1$.
This function is illustrated on Figure~\ref{fig:pre_fastgen}.

\begin{figure}
    \begin{center}
    \begin{tikzpicture}[xscale=0.13,yscale=0.05]
        \draw[->] (0.9,0) -- (100.1,0) node[right] {$t$};
        \draw[->] (1,-0.1) -- (1,115.1);
        \draw[red,domain=1:99,samples=1000] plot function{1/(2-cos(x)-cos(0.6971153242*x))};
        \draw[blue,densely dotted,smooth] plot file{lmcs.prefastgen_g.table};
        \node[draw=gray!75!black,fill=gray!15,rounded corners=3pt,text width=5cm,anchor=north west] (nodespike) at (5,85)
            {Sequence of \textbf{arbitrarily}\\\textbf{growing} spikes.
            \emph{But the distance between them increases as well.}};
        \draw[gray!75!black,->] (nodespike) -- (62,88);
        \draw[gray!75!black,->] (nodespike) -- (19,10);
    \end{tikzpicture}
    \end{center}
    \caption{Illustration of $g$ (in dotted blue) from Lemma~\ref{lem:pre_fastgen}: we start from a function $f$ (in red) that spikes
        and then integrate it to make it increasing.\label{fig:pre_fastgen}}
\end{figure}
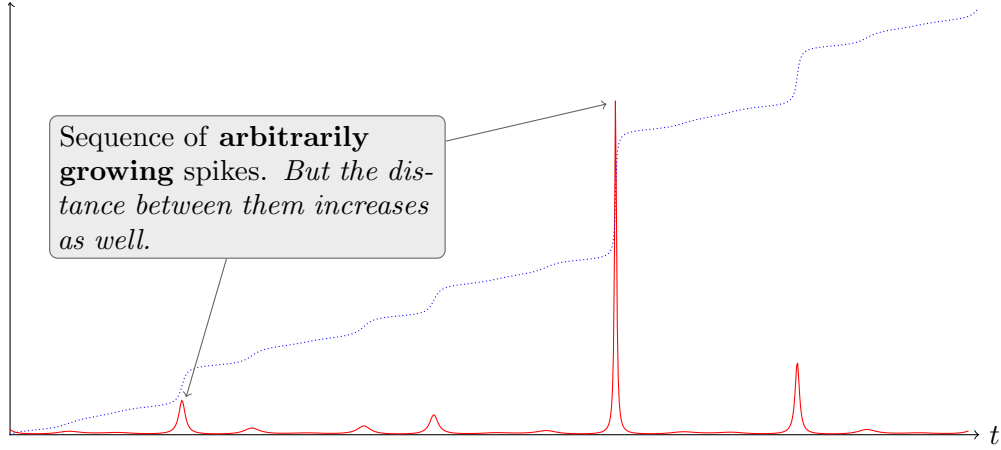

\begin{lemC}[\cite{Bank75}]\label{lem:pre_fastgen}
There exists a positive generable function $g$ and an absolute constant $c>0$
such that for any increasing sequence $a\in\N^\N$ with $a_n\geqslant2$ for all $n$, there exists
$\alpha\in\R$ such that $g(\alpha,\cdot)$ is defined over $[1,\infty)$, nondecreasing and for
any $n\in\N$ and $t\geqslant 2\pi b_n$, $g(\alpha,t)\geqslant ca_n$
where $b_n=\prod_{k=0}^{n-1}a_k$.
Furthermore, the map $a\mapsto\alpha$ is $(\nu_\N^\omega,\rho)$-computable.
\end{lemC}

\begin{proof}
We give a sketch of the proof, following the
presentation from \cite{Bank75}.
For any $\alpha\in\R$ and $t>0$, let
\[f(\alpha,t)=\frac{1}{2-\cos(t)-\cos(\alpha t)}.\]
Since $\sin$ and $\cos$ are generable, it follows that $f$ is generable because
it has a connected domain of definition. Indeed, $f(\alpha,t)$ is well-defined
except on $$X=\set{(\alpha,2k\pi):\alpha\in\Q,k\in\N,k\alpha\in\N}$$ which is a totally
disconnected set in $\R^2$.
Let
\begin{equation}\label{eq:pre_fastgen:def_alpha}
\alpha_a=\sum_{n=1}^\infty b_n^{-1}
\qquad\text{where }b_n=\prod_{k=0}^{n-1}a_k
\end{equation}
which is well-defined if $a_n$ is a strictly increasing sequence. Indeed, it implies
that $b_n\geqslant (n-1)!$ and $\alpha_a\leqslant\sum_{n=0}^\infty\tfrac{1}{n!}=e$.
One can easily show (by contradiction for example) that $\alpha_a$ must be irrational. Also define
\[g(\alpha,t)=\int_1^tf(\alpha,u)du\]
which is generable. Let $n\in\N$,
define $\delta_n=\sum_{k=n+1}^{\infty}\tfrac{b_n}{b_k}$.
Let $t\in[2\pi(b_n-\delta_n),2\pi b_n]$, write $\varepsilon=2\pi b_n-t\in[0,2\pi\delta_n]$ and observe that
\[1-\cos(t)=1-\cos(t-2\pi b_n)=1-\cos(\varepsilon)\leqslant \varepsilon^2\leqslant 4\pi^2\delta_n^2.\]
Furthermore,  and note that
\begin{align*}
1-\cos(\alpha t)
    &=1-\cos(2\pi\alpha b_n-\varepsilon)\\
    &=1-\cos\left(2\pi\sum_{k=0}^{n}\tfrac{b_n}{b_k}
        +2\pi\sum_{k=n+1}^{\infty}\tfrac{b_n}{b_k}-\varepsilon\right)\\
    &=1-\cos\left(2\pi\sum_{k=n+1}^{\infty}\tfrac{b_n}{b_k}-\varepsilon\right)&&\text{since }\tfrac{b_n}{b_k}\in\N\text{ for }k\leqslant n\\
    &=1-\cos\left(2\pi\delta_n-\varepsilon\right)\\
    &\leqslant \left(2\pi\delta_n-\varepsilon\right)^2&&\text{since }1-\cos(x)\leqslant x^2\\
    &\leqslant \left(2\pi\delta_n\right)^2&&\text{since }\varepsilon\leqslant 2\pi\delta_n.
\end{align*}
It follows that
\[f(\alpha_a,t)
    =(1-\cos(t)+1-\cos(\alpha t))^{-1}
    \geqslant(4\pi^2\delta_n^2+4\pi^2\delta_n^2)^{-1}
    \geqslant\tfrac{1}{8\pi^2\delta_n^2}.
\]
Thus
\begin{align*}
g(\alpha_a,2\pi b_n)
    &=\int_1^{2\pi b_n}f(\alpha_a,t)dt\\
    &\geqslant\int_{2\pi (b_n-\delta_n)}^{2\pi b_n}f(\alpha_a,t)dt&&\text{since $f$ is positive}\\
    &\geqslant\int_{2\pi (b_n-\delta_n)}^{2\pi b_n}\tfrac{1}{8\pi^2\delta_n^2}dt\\
    &=\tfrac{\delta_n}{8\pi^2\delta_n^2}=\tfrac{1}{8\pi^2\delta_n}.
\end{align*}
But note that
\begin{align*}
\delta_n
    &=\sum_{k=n+1}^{\infty}\tfrac{b_n}{b_k}
    \leqslant \sum_{k=n+1}^{\infty}a_n^{n-k}
        &&\text{since }\tfrac{b_n}{b_k}=(a_n\cdots a_{k-1})^{-1}\\
    &=\frac{a_n^{-1}}{1-a_n^{-1}}
    \leqslant2a_n^{-1}&&\text{since }a_n\geqslant2.
\end{align*}
It then easily follows that
\[g(\alpha_a,2\pi b_n)\geqslant \tfrac{a_n}{16\pi^2}\]
and the result follows from the fact that $g$ is nondecreasing.

The computability of the map $a\mapsto\alpha$ is the only missing result. It is immediate from
\eqref{eq:pre_fastgen:def_alpha} that the map $(a,n)\mapsto b_n$ is $([\nu_\N^\omega,\nu_\N],\nu_\N)$-computable
since each $b_n$ is a product of finitely many $a_i$. Furthermore, $b_n\geqslant(n-1)!$ thus for
any $n\geqslant 1$,
\[\left|\alpha_a-\sum_{i=1}^nb_i^{-1}\right|
    \leqslant\sum_{i\geqslant n-1}\frac{1}{(i-1)!}
    \leqslant\sum_{i\geqslant n}\frac{1}{i!}
    \leqslant\sum_{i\geqslant 0}\frac{1}{n!2^i}
    \leqslant\tfrac{2}{n!}\leqslant 2^{2-n}.
\]
It follows that $\left(\sum_{i=1}^nb_i^{-1}\right)_{i\in\N}$ is a Cauchy sequence of $\alpha_a$
of known convergence rate. It suffice to note that $(b,n)\mapsto\sum_{i=1}^nb_i^{-1}$ is
$([\nu_\N^\omega,\nu_\N],\rho)$-computable, since it only involves a finite number of sum and inverses
of real numbers.
\end{proof}

\begin{rem}
    As noted earlier, Lemma~\ref{eq:pre_fastgen:def_alpha} and the Theorem~\ref{th:fastgen}
    build \emph{partial functions}. Indeed we only show that the solution exists at all
    times $t\in\Rp$ for certain well-chosen $\alpha$. In particular, it can be easily checked
    that the ODE in Lemma~\ref{eq:pre_fastgen:def_alpha} explodes in finite time for all rational
    $\alpha$.
\end{rem}

Essentially, Lemma \ref{lem:pre_fastgen}
proves that there exists a function $g$ such that for any
$n\in\N$, $g(\alpha,a_0a_1\cdots a_{n-1})\geqslant a_n$. Note that this is not quite
what we are aiming for: the function $g$ is indeed $\geqslant a_n$ but at times
$a_0a_1\cdots a_{n-1}$ instead of $n$. Since $a_0a_1\cdots a_{n-1}$ is a very big number
compared to $n$, $g$ does not grow fast enough for our needs. The idea is to ``accelerate''
$g$ by composing it with a fast growing function $h$, ideally such that $h(n)\geqslant a_0\cdots a_{n-1}$.
This would ensure that $g(h(n))\geqslant n$. This is a chicken-and-egg
problem because to build such a function $h$, we need to build a fast growing function!
We now try to explain how to solve this problem.

Fix a sequence $(a_n)_n$ and let
$g$ be the function from Lemma~\ref{lem:pre_fastgen} and
$\alpha_a$ be the parameter that corresponds to $a$ (we omit the $\alpha_a$ for readability so $g(x)=g(\alpha_a,x)$).
Consider the following sequence:
\[x_0=a_0,\qquad x_{n+1}=x_n g(x_n).\]
Then observe that
\[x_1=x_0 g(x_0)=a_0 g(a_0)\geqslant a_0a_1,
\qquad x_2=x_1 g(x_1)\geqslant a_0a_1 g(a_0a_1)\geqslant a_0a_1a_2,
\quad\ldots\]
It is not hard to see that $x_n\geqslant a_0a_1\cdots a_n\geqslant a_n$. We then
use our generable gadget of Section~\ref{sec:helper_gen} to simulate this discrete
sequence with a differential equation. Intuitively, we build a differential equation
such that the solution $y$ satisfies $y(n)\approx x_n$. More precisely, we use two
variables $y$ and $z$ such that over $[n,n+1/2]$, $z'\approx 0$ and $y(t)\rightarrow z g(z)$
and over $[n+1/2,n+1]$, $y'\approx 0$ and $z(t)\rightarrow y$. Then if $y(n)\approx z(n)\approx x_n$
then $y(n+1)\approx z(n+1)\approx x_{n+1}$.

\begin{thm}\label{th:fastgen}
There exists $\Gamma\subseteq\R$ and a positive uniformly-generable function $\fastgen:\Gamma\times\Rp\to\R$
such that for any $x\in\Rp^\N$, there exists $\alpha\in\Gamma$ such that for any $n\in\N$ and $t\in\Rp$,
\[\fastgen(\alpha;t)\geqslant x_n\qquad\text{if }t\geqslant n.\]
Furthermore, $\fastgen(\alpha;\cdot)$ is nondecreasing.
In addition, the map $x\mapsto\alpha$ is $(\rho^\omega,\rho)$-computable.
\end{thm}

\begin{proof}

Let $\delta=4$.
Apply Lemma~\ref{lem:pre_fastgen} to get $g$ and $c$. Let $a\in\N^N$ be an increasing sequence
such that $a_n\geqslant x_n$,
then there exists $\alpha_x\in\R$ such that for all $n\in\N$,
\begin{equation}\label{eq:fastgen:g}
g(\alpha_x,t)\geqslant ca_n
\end{equation}
for all $t\geqslant 2\pi b_n$ where $b_n=\prod_{k=0}^{n-1}a_k$. Consider the following
system of differential equations, for $\phi=2$,
\[
\left\{\begin{array}{@{}r@{\hspace{.25em}}l@{}}
y(0)&=\delta+2\pi\\z(0)&=\delta+2\pi
\end{array}\right.
,\qquad
\left\{\begin{array}{@{}r@{\hspace{.25em}}l@{}}
y'(t)&=\pereach\left(t,\phi,y(t),\delta+\tfrac{1}{c}z(t)g(\alpha_x,z(t))\right)\\
z'(t)&=\pereach\left(t+\tfrac{1}{2},\phi,z(t),1+y(t)\right)
\end{array}\right..
\]
Apply Theorem~\ref{th:pereach} to show that $y$ and $z$ exist over $\Rp$.
We will show the following result by induction on $n\in\N$:
\begin{equation}\label{eq:fastgen:induc}
\min(y(n),z(n))\geqslant\delta+2\pi b_n.
\end{equation}
The result is trivial for $n=0$ since $y(0)=z(0)=\delta+2\pi=\delta+2\pi b_0\geqslant 1+2\pi b_0$. Let $n\in\N$
and assume that \eqref{eq:fastgen:induc} holds for $n$. Apply Theorem~\ref{th:pereach} (item~\ref{th:pereach:constant})
to $z$ to get that for any $t\in[n,n+\tfrac{1}{2}]$,
\[|z(t)-z(n)|
    \leqslant\int_n^t\exp\left(-\phi\right)du
    \leqslant\int_n^{n+\tfrac{1}{2}}\exp\left(-\phi\right)du\leqslant\tfrac{1}{2}e^{-\phi}\leqslant1.
\]
In particular, it follows that for any $t\in[n,n+\tfrac{1}{2}]$,
\begin{align*}
    z(t)&\geqslant z(n)-1
        \geqslant 2\pi b_{n}&&\text{since }z(n)\geqslant1+2\pi b_{n}\\
    g(\alpha_x,z(t))&\geqslant c a_{n}&&\text{using \eqref{eq:fastgen:g}}\\
    z(t)g(\alpha_x,z(t))&\geqslant  2\pi b_nca_{n}&&\text{since }z(t)\geqslant2\pi b_{n}\\
        &=2\pi cb_{n+1}&&\text{since }b_{n+1}=b_na_n\\
    \delta+\tfrac{1}{c}z(t)g(\alpha_x,z(t))&\geqslant\delta+2\pi b_{n+1}.
\end{align*}
Note that $\phi\geqslant1$ then
apply Theorem~\ref{th:pereach} (item~\ref{th:pereach:lower_conv}) to $y$ using the above inequality to get that
\[y(n+\tfrac{1}{2})
    \geqslant \delta+2\pi b_{n+1}-e^{-1}
    \geqslant \delta-1+2\pi b_{n+1}
\]
and (item~\ref{th:pereach:lower_bound}) for any $t\in[n,n+\tfrac{1}{2}]$,
\begin{equation}\label{eq:fasten:y_ineq_1}
y(t)\geqslant\min(y(n),\delta+2\pi b_{n+1})
    \geqslant\min(1+2\pi b_n,\delta+2\pi b_{n+1})
    \geqslant1+2\pi b_n,
\end{equation}
and (item~\ref{th:pereach:constant}) for any $t\in[n+\tfrac{1}{2},n+1]$,
\begin{equation}\label{eq:fasten:y_ineq_2}
|y(t)-y(n+\tfrac{1}{2})|\leqslant\int_{n+\tfrac{1}{2}}^te^{-\phi} du\leqslant\tfrac{e^{-\phi}}{2}\leqslant 1.
\end{equation}
Thus $y(t)\geqslant y(n+\tfrac{1}{2})-e^{-1}\geqslant \delta-2-2\pi b_{n+1}$ for any $t\in[n+\tfrac{1}{2},n+1]$.
Note that $\phi\geqslant1$ and
apply Theorem~\ref{th:pereach} (item~\ref{th:pereach:lower_conv}) to $z$ using the above inequality to get that
\[z(n+1)
    \geqslant \min_{u\in[n+\tfrac{1}{2},n+1]}y(u)-e^{-1}
    \geqslant \delta-2+2\pi b_{n+1}-e^{-1}
    \geqslant \delta-3+2\pi b_{n+1}.
\]
And since $\delta-3\geqslant1$, we have shown that $y(n+1)$ and $z(n+1)$ are greater than $1+2\pi b_{n+1}$.
Furthermore, \eqref{eq:fasten:y_ineq_1} and \eqref{eq:fasten:y_ineq_2} prove that
for any $t\in[n,n+1]$,
\[y(t)\geqslant\min(1+2\pi b_n,\delta-2+2\pi b_{n+1})\geqslant1+2\pi b_n.\]
We can thus let $\fastgen(\alpha;t)=y(1+t)$ and get the result since $1+2\pi b_{n+1}\geqslant a_n$.
Finally, $(\alpha_x;t)\mapsto y(t)$ is uniformly-generable by Theorem~\ref{th:unif_gpac_ext_ivp_stable}
because $\pereach$ and $g$ are generable and the initial condition
is computable.

The computability of the map $x\mapsto\alpha$ follows from the computability of the map $a\mapsto\alpha$
(Lemma~\ref{lem:pre_fastgen}) and the map $x\mapsto a$. Note that the only condition which $a\in\N^\N$ has
to satisfy is $a_n\geqslant x_n$. Given a real number represented by its Cauchy sequence,
\emph{with a known rate of convergence}, it is trivial to compute an integer upper bound on this number.
\end{proof}

\section{Generating a sequence of dyadic rationals}
\label{sec:dyadic}

A major part of the proof requires to build a function to approximate arbitrary
numbers over intervals $[n,n+1]$. Ideally we would like to build a function that gives
$x_0$ over $[0,1]$, $x_1$ over $[1,2]$, etc. Before we get there, we solve a somewhat
simpler problem by making a few assumptions:
\begin{itemize}
\item we only try to approximate dyadic numbers, \emph{i.e.} numbers of the form $m2^{-p}$,
    and furthermore we only approximate  with error $2^{-p-3}$;
\item if a dyadic number has size $p$, meaning that it can be written as $m2^{-p}$
    but not $m'2^{-p+1}$ then it will take a time interval of $p$ units to approximate:
    $[k,k+p]$ instead of $[k,k+1]$;
\item the function will only approximate the dyadics over intervals $[k,k+\tfrac{1}{2}]$
    and not $[k,k+1]$.
\end{itemize}
This processus is illustrated in Figure~\ref{fig:dygen}: given a sequence $d_0,d_1,\ldots$
of dyadics, there is a corresponding sequence $a_0,a_1,\ldots$ of times such that
the function approximates $d_k$ over $[a_k,a_k+\tfrac{1}{2}]$ within error $2^{-p_k}$
where $p_k$ is the size of $d_k$. The theorem contains an explicit formula for
$a_k$ that depends on some absolute constant $\delta$.

Figure~\ref{fig:dygen} highlights a feature of $\dygen$: it is an almost piecewise constant
function. However we only control the values it takes over small intervals $[a_i,a_i+\tfrac{1}{2}]$,
and we have no idea what is the value the rest of the time (even if we know that it is
almost piecewise constant).
\bigskip


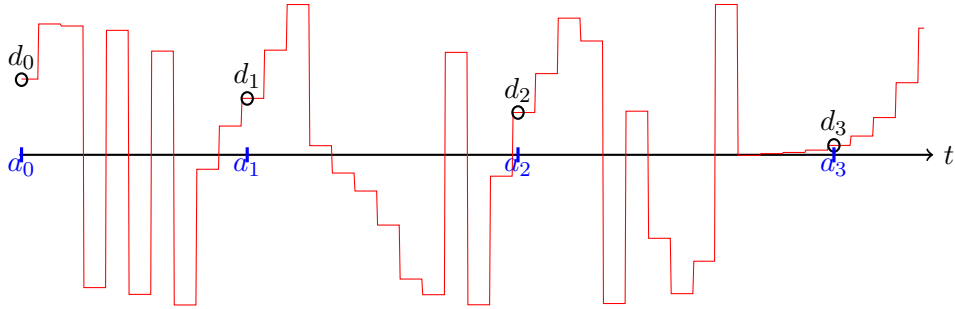
\begin{figure}[h]
\begin{center}
\begin{tikzpicture}[domain=0:40,samples=1000,xscale=0.3,yscale=2]
\draw[thick,->] (-0.1,0) -- (40.4,0) node[right] {$t$};
\draw[color=red] plot[id=dygen] function{sin(2*0.08404422889012441721945378958480432629585*pi*exp(log(2)*floor(x-1/4.+1/2.)))};
\foreach \idx/\x/\y in {0/0/0.5, 1/10/0.375, 2/22/0.28125, 3/36/0.0625}
{
    \draw[color=black,thick] (\x,\y) circle[x radius=0.25,y radius=0.25/6] node[above] {$d_\idx$};
    \draw[color=blue,very thick] (\x,-0.05) -- (\x,0.05) node[below=0.1] {$a_\idx$};
}
\end{tikzpicture}
\end{center}
\caption{Graph of $\dygentext$ for $d_0=2^{-1}$, $d_1=2^{-3}+2^{-1}$, $d_2=2^{-5}+2^{-2}$ and $d_3=2^{-4}$ (other values ignored)
assuming that $\delta=9$. We get that $a_0=0$, $a_1=10$, $a_2=22$, $a_3=36$.}
\label{fig:dygen}
\end{figure}

Let $\D_p=\set{m2^{-p}:m\in\set{0,1,\ldots,2^p-1}}$
and $\D=\bigcup_{n\in\N}\D_p$ denote the set of dyadic rationals in $[0,1)$. For any $q\in\D$,
we define its \emph{size} by $\dysize(q)=\min\set{p\in\N:q\in\D_p}$.

\begin{thm}\label{th:dygen}
There exists $\delta\in\Nps$, $\Gamma\subseteq\R^2$ and a uniformly-generable function
$\dygen:\Gamma\times\Rp\to\R$ such that for any dyadic sequence
$q\in\D^\N$, there exists $(\alpha,\beta)\in\Gamma$ such that for any  $n\in\N$,
\[|\dygen(\alpha,\beta;t)-q_n|\leqslant 2^{-\dysize(q_n)-3}\qquad\text{for any }t\in [a_{n},a_n+\tfrac{1}{2}]\]
where $a_n=\sum_{k=0}^{n-1}(\dysize(q_k)+\delta)$.
Furthermore, $|\dygen(\alpha,\beta;t)|\leqslant 1$ for all $\alpha,\beta$ and $t$.
In addition, the map $q\mapsto(\alpha,\beta)$ is $(\nu_\Q^\omega,[\rho,\rho])$-computable.
\end{thm}

\begin{lem}\label{lem:dyadic_sine}
For any $q\in\D_p$, there exists $q'\in\D_{p+3}$ such that $|\sin(2q'\pi)-q|\leqslant 2^{-p}$
and $|q'|\leqslant 1$.
Furthermore, $x\mapsto\sin(2\pi x)$ is $8$-Lipschitz.
In addition, the map $q\mapsto q'$ is ($\nu_\Q,\nu_\Q)$-computable.
\end{lem}

\begin{proof}
Let $f(x)=\sin(2\pi x)$ for $x\in[0,\tfrac{1}{4}]$. Clearly $f$ is surjective
from $[0,\tfrac{1}{4}]$ to $[0,1]$ thus there exists $x'\in[0,\tfrac{1}{4}]$
such that $f(x')=q$. Furthermore, since $f'(x)=2\pi\cos(2\pi x)$
and $2\pi\leqslant 8$, $f$ is $8$-Lipschitz. Let $q'=\lfloor 2^{p+3}x'\rfloor 2^{-p-3}$,
then $q'\in\D_{p+3}$ and $|q'-x'|\leqslant 2^{-p-3}$ by construction. Clearly $|q'|\leqslant 1$,
and furthermore,
\[|f(q')-q|
    \leqslant|f(q')-f(x')|+|f(x')-q|
    \leqslant 8|q'-x'|+0
    \leqslant 8\cdot2^{-p-3}
    \leqslant 2^{-p}.\]
Note that $f$ is not only surjective from $[0,\tfrac{1}{4}]$ to $[0,1]$ but also increasing and $8$-Lipschitz.
Furthermore, $f$ is $(\rho,\rho)$-computable thus a simple dichotomy is enough to find a suitable
\textbf{rational} $x'$. To conclude,
use the fact that the map $x'\mapsto q'$ is $(\nu_\Q,\nu_\Q)$-computable. Note that it is crucial
that $x'$ is rational because the floor function is not $(\rho,\rho)$-computable.
\end{proof}

\begin{proof}[Proof of Theorem~\ref{th:dygen}]
Let $\delta=9$.
Consider the function
\[f(\alpha,t)=\sin\left(2\alpha\pi 2^t\right)\]
defined for any $\alpha,t\in\R$. Then $f$ is generable because $\sin$ and $\exp$
are generable. For all $n\in\N$, note that
$q_n\in\D_{\dysize(q_n)}\subseteq \D_{\dysize(q_n)+\delta-3}$ and apply
Lemma~\ref{lem:dyadic_sine} to $q_n$ to get $q_n'\in\D_{\dysize(q_n)+\delta}$ such that
\begin{equation}\label{eq:dygen:sin}
|\sin(2q_n'\pi)-q_n|\leqslant 2^{-\dysize(q_n)-\delta+3}.
\end{equation}
Now define
\begin{equation}\label{eq:dyagen:alpha}
\alpha_q=\sum_{k=0}^\infty q_k'2^{-a_k}
    \qquad\text{where }a_k=\sum_{\ell=0}^{k-1}(\dysize(q_\ell)+\delta).
\end{equation}
It is not hard to see that $\alpha_q$ is well-defined (\emph{i.e.} the sum converges).
Let $n\in\N$, then
\begin{align*}
f(\alpha_q,a_n)
    &=\sin(2\pi\alpha_q 2^{a_n})\\
    &=\sin\left(2\pi\sum_{k=0}^\infty q_k'2^{-a_k+a_n}\right)\\
    &=\sin\left(2\pi\sum_{k=0}^{n-1}q_k'2^{-a_k+a_n}
        +2\pi q_n'
        +2\pi\sum_{k=n+1}^\infty q_k'2^{-a_k+a_n}\right).
\end{align*}
But for any $k\leqslant n-1$,
\[a_n-a_k=\sum_{\ell=k}^{n-1}(\dysize(q_\ell)+\delta)\geqslant\dysize(q_k)+\delta\]
and since $q_k'\in\D_{\dysize(q_k)+\delta}$ and $a_n-a_k\in\N$, it follows that $q_k'2^{-a_k+a_n}\in\N$.
Consequently,
\[f(\alpha_q,a_n)=\sin\left(2\pi q_n'+2\pi\sum_{k=n+1}^\infty q_k'2^{-a_k+a_n}\right)
    =\sin\left(2\pi(q_n'+u)\right)\]
where $u=\sum_{k=n+1}^\infty q_k'2^{-a_k+a_n}$. But for any $k\geqslant n+1$,
\[a_k-a_n
    =\sum_{\ell=n}^{k-1}(\dysize(q_\ell)+\delta)
    \geqslant\dysize(q_n)+\sum_{\ell=n}^{k-1}\delta
    =\dysize(q_n)+\delta(k-n).
\]
Consequently,
\[|u|
    \leqslant\sum_{k=n+1}^\infty |q_k'|2^{-a_k+a_n}
    \leqslant\sum_{k=n+1}^\infty 2^{-\dysize(q_n)-\delta(k-n)}
    \leqslant2^{-\dysize(q_n)}\sum_{k=1}^\infty 2^{-\delta k}
    \leqslant 2^{-\dysize(q_n)-\delta+1}.
\]
Since $x\mapsto\sin(2\pi x)$ is $8$-Lipschitz, it follows that
\begin{align}
|f(\alpha_q,a_n)-q_n|
    &=|\sin\left(2\pi(q_n'+u)\right)-q_n|\nonumber\\
    &\leqslant|\sin\left(2\pi(q_n'+u)\right)-\sin(2\pi q_n')|+|\sin(2q_n'\pi)-q_n|\nonumber\\
    &\leqslant8|u|+2^{-\dysize(q_n)-\delta+3}\nonumber&&\text{using \eqref{eq:dygen:sin}}\\
    &\leqslant8\cdot 2^{-\dysize(q_n)-\delta+1}+2^{-\dysize(q_n)-\delta+3}\nonumber\\
    &\leqslant2^{-\dysize(q_n)-\delta+5}\label{eq:dygen:f}.
\end{align}
Recall that $\rnd$ is the generable rounding function from Lemma~\ref{lem:rnd},
and $\fastgentext$ the fast growing function from Theorem~\ref{th:fastgen}.
Let $\alpha,\beta,t\in\R$, if $\fastgen(\beta;t)$ exists then let
\[\dygen(\alpha,\beta;t)=f(\alpha,r(\beta;t))\]
where
\[r(\beta;t)=\rnd\left(t-\tfrac{1}{4},\fastgen(\beta;t)\ln2,4\right).\]
Note that $\dygentext$ is uniformly-generable because $f,\rnd$ and $\fastgentext$ are uniformly-generable.
Apply\footnote{Technically speaking, we apply it to the sequence $x_n=a_k$ if
    $n=a_k+\dysize(q_n)+\delta$, and $x_n=0$ otherwise.}
Theorem~\ref{th:fastgen} to get $\beta_q\in\R$ such that for any $n\in\N$ and $t\in\Rp$,
\[\fastgen(\beta_q;t)\geqslant a_n+\dysize(q_n)+\delta\qquad\text{if }t\in[a_n,a_n+1).\]
Let $n\in\N$ and $t\in[a_n,a_n+\tfrac{1}{2}]$, then
$t-\tfrac{1}{4}\in\left[a_n-\tfrac{1}{2}+\tfrac{1}{\lambda},a_n+\tfrac{1}{2}-\tfrac{1}{\lambda}\right]$
for $\lambda=4$. Thus we can apply Lemma~\ref{lem:rnd} and get that
\begin{equation}\label{eq:dygen:r}
|r(\beta_q,t)-a_n|\leqslant e^{-\fastgen(\beta_q;t)\ln2}=2^{-\fastgen(\beta_q;t)}
    \leqslant 2^{-a_n-\dysize(q_n)-\delta}\leqslant1.
\end{equation}
Observe that
\[\left|\tfrac{\partial f}{\partial t}(\alpha_q,t)\right|
    =2\pi|\alpha_q|2^t|\cos(2\alpha_q\pi 2^t)|
    \leqslant 2\pi2^t\leqslant 2^{t+3}.\]
Thus for any $t,t'\in\R$,
\[|f(\alpha_q,t)-f(\alpha_q,t')|\leqslant 2^{3+\max(t,t')}|t-t'|.\]
It follows that for any $n\in\N$ and $t\in[a_n,a_n+\tfrac{1}{2}]$,
\begin{align*}
|\dygen(\alpha_q,\beta_q;t)-q_n|
    &=|f(\alpha_q,r(\beta_q;t))-q_n|\\
    &\leqslant |f(\alpha_q,r(\beta_q;t))-f(\alpha_q,a_n)|+|f(\alpha_q,a_n)-q_n|\\
    &\leqslant 2^{3+\max(a_n,r(\beta_q;t))}|r(\beta_q;t)-a_n|+2^{-\dysize(q_n)-3}&&\text{using \eqref{eq:dygen:f}}\\
    &\leqslant 2^{3+a_n+1}2^{-\dysize(q_n)-a_n-\delta}+2^{-\dysize(q_n)-\delta+5}&&\text{using \eqref{eq:dygen:r}}\\
    &\leqslant 2^{-\dysize(q_n)-\delta+4}+2^{-\dysize(q_n)-\delta+5}\\
    &\leqslant 2^{-\dysize(q_n)-\delta+6}\\
    &\leqslant 2^{-\dysize(q_n)-3}&&\hspace{-2.5em}\text{since }\delta-6\geqslant3.
\end{align*}
To see that the map $q\mapsto(\alpha_q,\beta_q)$ is computable, first note that the map $q_n\mapsto q_n'$
is computable (Lemma~\ref{lem:dyadic_sine}), thus the map $q\mapsto q'$ is
$(\nu_\Q^\omega,\nu_\Q^\omega)$-computable. It is clear from \eqref{eq:dyagen:alpha} that $q'\mapsto a$
is also computable. Using a similar argument as above, one can easily see that the partial sums
(of the infinite sum) defining $\alpha_q$ in \eqref{eq:dyagen:alpha} form a Cauchy sequence with
convergence rate $k\mapsto 2^{-k}$ because $a_k\geqslant k\delta\geqslant k$. Finally, $q\mapsto \beta_q$
is computable by Theorem~\ref{th:fastgen}.
\end{proof}

\section{Generating a sequences of bits}
\label{sec:bits}

We saw in the previous section how to generate a dyadic generator. Unfortunately,
we saw that it generates dyadic $d_n$ at times $a_n$, whereas we would like to get $d_n$
at time $n$ for our approximation. Our approach is to build a signal generator that
will be high exactly at times $a_n$. Each time the signal will be high, the system will copy
the value of the dyadic generator to a variable and wait until the next signal.
Since the signal is binary, we only need to generate a sequence of bits.
Note that this theorem has a different flavour from the dyadic generator: it generates
a more restrictive set of values (bits) but does so much better because we have
control over the timing and we can approximate the bits with arbitrary precision.

Figure~\ref{fig:bitgen} shows what $\bitgen$ looks like: it is an almost piecewise constant
function such that the value in the interval $[n,n+\tfrac{1}{2}]$ is almost the $n^{th}$
digit of $\alpha$.
\bigskip

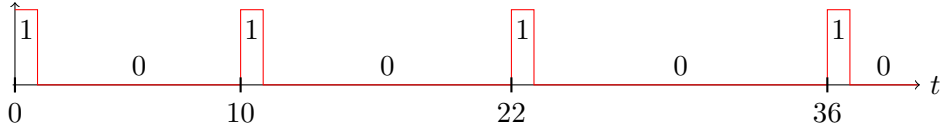
\begin{figure}[h]
    \begin{center}
    \begin{tikzpicture}[xscale=0.3,yscale=1]
        \draw[white] (-1,0) -- (11.2,0);
        \draw[->] (-0.1,0) -- (40.1,0) node[right] {$t$};
        \draw[->] (0,-0.1) -- (0,1.1);
        \draw[red] (0,1) -- ++(1,0) node[midway,below,black] {1}
            -- ++(0,-1) --++(9,0) node[midway,above,black] {0}
            -- ++(0,1) -- ++(1,0) node[midway,below,black] {1}
            -- ++(0,-1) --++(11,0) node[midway,above,black] {0}
            -- ++(0,1) -- ++(1,0) node[midway,below,black] {1}
            -- ++(0,-1) --++(13,0) node[midway,above,black] {0}
            -- ++(0,1) -- ++(1,0) node[midway,below,black] {1}
            -- ++(0,-1) --++(3,0) node[midway,above,black] {0};
        \foreach \x in {0,10,22,36} {
            \draw[thick] (\x,0.1) -- (\x,-0.1) node[below] {\x};
        }
    \end{tikzpicture}
    \end{center}
    \caption{(Ideal) graph of $\bitgen$ for $b\in\set{0,1}^\N$ where $b_0=b_{10}=b_{22}=b_{36}=1$ and all the
        other bits are $0$. \label{fig:bitgen}}
\end{figure}

\begin{rem}Although it is possible to define $\bitgentext$ using $\dygentext$, it
does not, in fact, give a shorter proof but definitely gives a more complicated
function.
\end{rem}

\begin{thm}\label{th:bitgen}
There exists $\Gamma\subseteq\R$ and a generable function $\bitgen:\Gamma\times\Rp^2\to\R$ such that for any bit sequence
$b\in\set{0,1}^\N$, there exists $\alpha_b\in\Gamma$ such that for any $\mu\in\Rp$,
$n\in\N$ and $t\in[n,n+\tfrac{1}{2}]$,
\[|\bitgen(\alpha_b,\mu,t)-b_n|\leqslant e^{-\mu}.\]
Furthermore, $|\bitgen(\alpha,\mu,t)|\leqslant 1$ for all $\alpha,\mu$ and $t$.
Finally, the map $b\mapsto\alpha_b$ is $(\nu_\N^\omega,\rho)$-computable.
\end{thm}

\begin{proof}

Consider the function
\[f(\alpha,t)=\sin\left(2\pi\alpha 4^t+\tfrac{4\pi}{3}\right)\]
defined for any $\alpha,t\in\R$. Then $f$ is generable because $\sin$ and $\exp$
are generable. For any $b\in\set{0,1}^\N$, let
\begin{equation}\label{eq:bitgen:def_alpha}
\alpha_b=\sum_{k=0}^\infty 2b_k4^{-k-1}.
\end{equation}
Let $b\in\set{0,1}^\N$ and $n\in\N$, observe that
\begin{align*}
f(\alpha_b,n)
    &=\sin\left(2\pi\sum_{k=0}^\infty 2b_k4^{-k-1}4^n+\tfrac{4\pi}{3}\right)\\
    &=\sin\left(2\pi\underbrace{\sum_{k=0}^{n-1} 2b_k4^{n-k-1}}_{\in\N}+2\pi2b_n4^{-1}+
        2\pi\sum_{k=n+1}^\infty 2b_k4^{n-k-1}+\tfrac{4\pi}{3}\right)\\
    &=\sin\left(\pi b_n+\tfrac{4\pi}{3}+\delta\right)
\end{align*}
where
\[\delta
    =2\pi\sum_{k=n+1}^\infty 2b_k4^{n-k-1}
    =4\pi\sum_{k=0}^\infty b_k4^{-k-2}
    \leqslant 4\pi\sum_{k=0}^\infty 4^{-k-2}
    =4\pi\tfrac{4^{-2}}{3}=\frac{\pi}{3}.
\]
It follows that if $b_n=0$ then $\pi b_n+\tfrac{4\pi}{3}+\varepsilon\in\left[\tfrac{4\pi}{3},\tfrac{5\pi}{3}\right]$
and if $b_n=1$ then $\pi b_n+\tfrac{4\pi}{3}+\varepsilon\in\left[\tfrac{7\pi}{3},\tfrac{8\pi}{3}\right]$.
Thus
\begin{equation}\label{eq:bitgen:f}
f(\alpha_b,n)\in\left[-1,-\tfrac{\sqrt{3}}{2}\right]\text{ if }b_n=0,
\qquad\qquad f(\alpha_b,n)\in\left[\tfrac{\sqrt{3}}{2},1\right]\text{ if }b_n=1.
\end{equation}
Furthermore,
\begin{equation}\label{eq:bitgen:alpha}
|\alpha_b|
    =\sum_{k=0}^\infty 2b_k4^{-k-1}
    \leqslant\tfrac{2}{4}\sum_{k=0}^\infty 4^{-k}
    \leqslant\tfrac{2}{3}.
\end{equation}
Let $\varepsilon\in\left[-\tfrac{1}{2},\tfrac{1}{2}\right]$, then
for any $t\in\R$,
\begin{align}
|f(\alpha_b,t+\varepsilon)-f(\alpha_b,t)|
    &=\left|\sin\left(2\pi\alpha_b 4^{t+\varepsilon}+\tfrac{4\pi}{3}\right)
            -\sin\left(2\pi\alpha_b 4^t+\tfrac{4\pi}{3}\right)\right|\nonumber\\
    &=\left|2\cos(2\pi\alpha_b(4^t+4^{t+\varepsilon})+\tfrac{8\pi}{3})
        \sin\left(2\pi\alpha_b(4^{t+\varepsilon}-4^t)\right)\right|\nonumber\\
    &\leqslant2|\sin(2\pi\alpha_b(4^{t+\varepsilon}-4^t))|\nonumber\\
    &\leqslant4\pi|\alpha_b||4^{t+\varepsilon}-4^t|\nonumber\\
    &\leqslant4\pi\frac{2}{3}4^t|4^\varepsilon-1|&&\text{using \eqref{eq:bitgen:alpha}}\nonumber\\
    &\leqslant\frac{8\pi}{3}4^t2\varepsilon&&\hspace{-3.8em}\text{using that }|\varepsilon|\leqslant\tfrac{1}{2}\nonumber\\
    &\leqslant 4^tB\varepsilon\label{eq:bitgen:mod_cont_f}
\end{align}
for some constant $B>0$.
Recall that $\rnd$ is the generable rounding function from Lemma~\ref{lem:rnd}.
Let $\alpha,t\in\R$, $\mu\in\Rp$ and define
\[g(\alpha,t)=f(\alpha,r(t))
\quad\text{where}\quad
r(t)=\rnd(t-\tfrac{1}{4},t\ln4+\ln B,4).\]
Note again that $g$ is generable because $f$ and $\rnd$ are generable.
Let $n\in\N$ and $t\in[n,n+\tfrac{1}{2}]$, then
$t-\tfrac{1}{4}\in\left[n-\tfrac{1}{4},n+\tfrac{3}{4}\right]
=\left[n-\tfrac{1}{4}+\tfrac{1}{\lambda},n+\tfrac{1}{2}-\tfrac{1}{\lambda}\right]$
for $\lambda=4$. Thus we can apply Lemma~\ref{lem:rnd} and get that
\[|r(t)-n|\leqslant e^{-(n+1)\ln4+\ln B}=\tfrac{4^{-n-1}}{B}.\]
It follows using \eqref{eq:bitgen:mod_cont_f} that
\[|g(\alpha_b,t)-f(\alpha_b,n)|
    =|f(\alpha_b,r(t))-f(\alpha_b,n)|
    \leqslant4^tB\tfrac{4^{-n-1}}{B}=\tfrac{1}{4}.\]
And since $\tfrac{\sqrt{3}}{2}-\tfrac{1}{4}\geqslant\tfrac{1}{2}$, we conclude
using \eqref{eq:bitgen:f} that for any $t\in\left[n,n+\tfrac{1}{2}\right]$,
\begin{equation}\label{eq:bitgen:g}
g(\alpha_b,t)\in\left[-1,-\tfrac{1}{2}\right]\text{ if }b_n=0,
\qquad\qquad g(\alpha_b,t)\in\left[\tfrac{1}{2},1\right]\text{ if }b_n=1.
\end{equation}
Finally, let $\alpha,t\in\R$, $\mu\in\Rp$ and define
\[\bitgen(\alpha,\mu\,t)=\frac{1+\tanh(2\mu g(\alpha,t))}{2}.\]
Note that $\bitgentext$ is generable because $\tanh$ and $g$ are generable. Let $\mu\in\Rp$,
$n\in\N$ and $t\in\left[n,n+\tfrac{1}{2}\right]$.
If $b_n=0$, then it follows from \eqref{eq:bitgen:g} that
\begin{align*}
    g(\alpha_b,n)&\leqslant-\tfrac{1}{2}\\
    |\tanh(2\mu g(\alpha_b,n))-\sgn(2\mu g(\alpha_b,n))|&\leqslant e^{-2\mu|g(\alpha_b,n)|}&&\text{using Lemma~\ref{lem:tanh}}\\
    |\tanh(2\mu g(\alpha_b,n))+1|&\leqslant e^{-\mu}\\
    |\bitgen(\alpha_b,\mu,t)-b_n|&\leqslant \tfrac{1}{2}e^{-\mu}\leqslant e^{-\mu}&&\text{since }b_n=0.
\end{align*}
Similarly, if $b_n=1$, then
\begin{align*}
    g(\alpha_b,n)&\geqslant\tfrac{1}{2}\\
    |\tanh(2\mu g(\alpha_b,n))-\sgn(2\mu g(\alpha_b,n))|&\leqslant e^{-2\mu|g(\alpha_b,n)|}\\
    |\tanh(2\mu g(\alpha_b,n))-1|&\leqslant e^{-\mu}\\
    |\bitgen(\alpha_b,\mu,t)-b_n|&\leqslant \tfrac{1}{2}e^{-\mu}\leqslant e^{-\mu}&&\text{since }b_n=1.
\end{align*}
Finally, it is clear from \eqref{eq:bitgen:def_alpha} that the partial sums are easily computable
and form a Cauchy sequence that converges at rate $4^{-k}$, thus $\alpha_b$ is computable from $b$.
\end{proof}

\section{Generating an almost piecewise constant function}
\label{sec:almost}

We have already explained the main intuition of this section in previous sections.
Using the dyadic generator and the bit generator as a signal, we can construct a
system that ``samples'' the dyadic at the right time and then holds
this value still until the next
dyadic. In essence, we just described an almost piecewise constant function. This function
still has a limitation: its rate of change is small so it can only approximate slowly
changing functions. Figure~\ref{fig:pwcgen} illustrates how this process
works at the high-level.

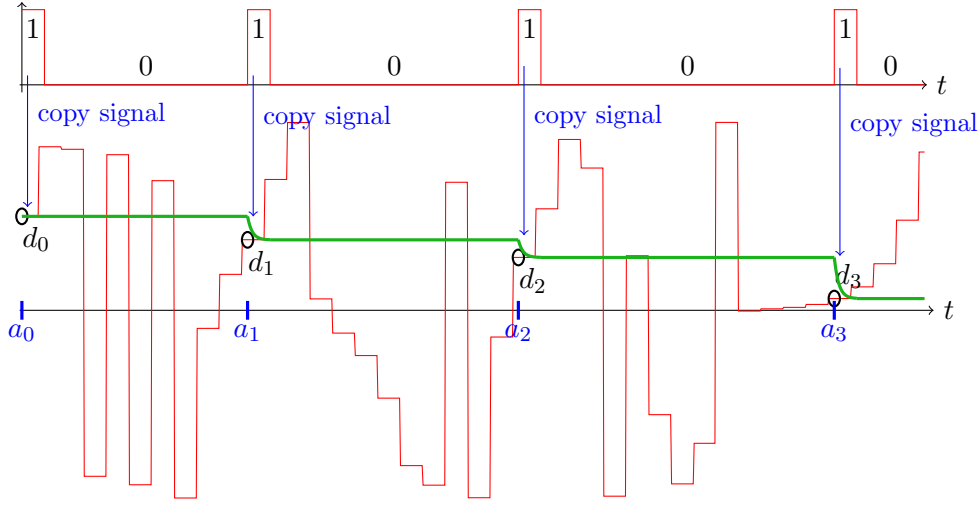
\begin{figure}[h]
    \begin{center}
    \begin{tikzpicture}[xscale=0.3,yscale=1,
            final signal/.style={very thick,darkgreen}
            ]
        \draw[white] (-1,0) -- (11.2,0);
        \begin{scope}
            \draw[->] (-0.1,0) -- (40.1,0) node[right] {$t$};
            \draw[->] (0,-0.1) -- (0,1.1);
            \draw[red] (0,1) -- ++(1,0) node[midway,below,black] {1}
                -- ++(0,-1) --++(9,0) node[midway,above,black] {0}
                -- ++(0,1) -- ++(1,0) node[midway,below,black] {1}
                -- ++(0,-1) --++(11,0) node[midway,above,black] {0}
                -- ++(0,1) -- ++(1,0) node[midway,below,black] {1}
                -- ++(0,-1) --++(13,0) node[midway,above,black] {0}
                -- ++(0,1) -- ++(1,0) node[midway,below,black] {1}
                -- ++(0,-1) --++(3,0) node[midway,above,black] {0};
        \end{scope}
        \begin{scope}[shift={(0,-3)},yscale=2.5]
                \draw[->] (-0.1,0) -- (40.4,0) node[right] {$t$};
                \draw[color=red,domain=0:40,samples=1000] plot function{
                    sin(2*0.08404422889012441721945378958480432629585*pi*exp(log(2)*floor(x-1/4.+1/2.)))};
                \foreach \idx/\x/\y/\p/\s in {0/0/0.5/below/0.5em, 1/10/0.375/below/0.5em, 2/22/0.28125/below/0.5em, 3/36/0.0625/above/.5em}
                {
                    \draw[color=black,thick] (\x,\y) circle[x radius=0.25,y radius=0.25/6] node[\p,xshift=\s] {$d_\idx$};
                    \draw[color=blue,very thick] (\x,-0.05) -- (\x,0.05) node[below=1ex] {$a_\idx$};
                }
                \draw[final signal] (0,0.5) -- (10,0.5);
                \draw[blue,<-] (0.25,0.55) -- ++(0,0.7) node[pos=0.7,right] {\small copy signal};
                \begin{scope}[shift={(10,0.5)}]
                \draw[final signal,domain=0:1,samples=50] plot function{(0.5-0.375)*(exp(-5*x)-1)};
                \end{scope}
                \draw[final signal] (11,0.375) -- (22,0.375);
                \draw[blue,<-] (10.25,0.5) -- ++(0,0.75) node[pos=0.7,right] {\small copy signal};
                \begin{scope}[shift={(22,0.375)}]
                \draw[final signal,domain=0:1,samples=50] plot function{(0.375-0.28125)*(exp(-5*x)-1)};
                \end{scope}
                \draw[final signal] (23,0.28125) -- (36,0.28125);
                \draw[blue,<-] (22.25,0.4) -- ++(0,0.9) node[pos=0.7,right] {\small copy signal};
                \begin{scope}[shift={(36,0.28125)}]
                \draw[final signal,domain=0:1,samples=50] plot function{(0.28125-0.0625)*(exp(-5*x)-1)};
                \end{scope}
                \draw[final signal] (37,0.0625) -- (40,0.0625);
                \draw[blue,<-] (36.25,0.29) -- ++(0,1.0) node[pos=0.7,right] {\small copy signal};
        \end{scope}
    \end{tikzpicture}
    \end{center}
    \caption{Illustration of the process to generate an almost piecewise-constant function: $\bitgen$ is used to
        generate a copy signal, which we synchronise with $\dygen$ to copy exactly the sequence of dyadic numbers
        we specified. In-between copies, we ensure that copied value does not change (sample and hold).\label{fig:pwcgen}}
\end{figure}

\begin{thm}\label{th:pwcgen}
There exists an absolute constant $\delta\in\N$, $p\in\N$, $\Gamma\subseteq\R^p$
and a uniformly-generable function $\pwcgen:\Gamma\times\Rp\to\R$
such that for any dyadic sequence $q\in\D^\N$, there exists $\alpha_q\in\Gamma$ such that for any $n\in\N$, putting
$a_n=\sum_{k=0}^{n-1}(\delta+\dysize(q_k))$, we have that
\[|\pwcgen(\alpha_q;t)-q_n|\leqslant 2^{-\dysize(q_n)}\qquad\text{for any }t\in [a_n+\tfrac{1}{2},a_{n+1}]\]
and
$\pwcgen(\alpha_q;t)\in I_n$ for any $t\in [a_n,a_n+\tfrac{1}{2}]$
where\footnote{With the convention that $[a,b]=[\min(a,b),\max(a,b)]$ and $I+\alpha J=\set{x+\alpha y:x\in I,y\in J}$.}
\[I_n:=\left[\pwcgen(\alpha_q;a_n),\pwcgen(\alpha_q;a_n+\tfrac{1}{2})\right]+2^{-\dysize(q_n)}[-1,1].\]
Finally, the map $q\mapsto\alpha_q$ is $(\nu_\Q^\omega,\rho^p)$-computable.
\end{thm}

\begin{proof}
Apply Theorem~\ref{th:dygen} to get $\delta\in\N$, $\dygentext$ uniformly-generable,
and $\alpha_q,\beta_q$ such that for any $n\in\N$,
\[|\dygen(\alpha_q,\beta_q;t)-q_n|\leqslant 2^{-\dysize(q_n)-3}\qquad\text{for any }t\in [a_{n},a_n+\tfrac{1}{2}]\]
where $a_n=\sum_{k=0}^{n-1}(\dysize(q_k)+\delta)$.
Let $b\in\set{0,1}^\N$ be the bit sequence defined by
\[b_n=\begin{cases}1&\text{if }n=a_k\text{ for some }k\\0&\text{otherwise.}\end{cases}\]
Apply Theorem~\ref{th:bitgen} to get $\bitgentext$ and $\gamma_b$ such that for any $\mu\in\Rp$,
$n\in\N$ and $t\in[n,n+\tfrac{1}{2}]$,
\[|\bitgen(\gamma_b,\mu,t)-b_n|\leqslant e^{-\mu}.\]
Let $x_n=\dysize(q_n)+3$ and apply Theorem~\ref{th:fastgen} to get $\lambda_a$ and $\fastgentext$
such that for any $n\in\N$,
\[\fastgen(\lambda_a;t)\geqslant x_n\qquad\text{for all }t\in[n,n+1].\]
Consider the following system for all $t\in\Rp$:
\[y(0)=q_0,\qquad y'(t)=\psi(t)r(t),\qquad
r(t)=\pereach(t,\phi(t),y(t),g(t))\]
where
\[\phi(t)=t+\fastgen(\lambda_a;t),
\qquad \psi(t)=2\bitgen(\gamma_b,\phi(t)+R(t),t),\]
\[g(t)=\dygen(\alpha_q,\beta_q;t),
\qquad R(t)=1+r(t)^2.\]
We omitted the parameters $\alpha,\beta,\gamma,\lambda$ in the functions $g$ and $\phi$
to make it more readable. It is clear that $g$ and $\phi$ are uniformly-generable since $\bitgentext,\fastgentext,\dygentext$
are uniformly-generable. It follows that\footnote{Technically, we have to include $q_0$ in the
parameters, even though the sequence $q$ is implicitly encoded in other parameters. This is because
the initial condition must be $q_0$ for the proof to work.} $(q_0,\alpha,\beta,\gamma,\lambda,t)\mapsto y(t)$ is also uniformly-generable
by Theorem~\ref{th:unif_gpac_ext_ivp_stable}. Indeed, the only extra thing we need to check is that the initial
condition is a computable function of the parameters, which it is since we just need to extract $q_0$
from the list of parameters.

We will show the result by induction. Let $n\in\N$ and assume that $|y(a_n)-q_n|\leqslant 2^{-x_n}$.
Note that this is trivially satisfied for $n=0$ since $a_0=0$ and thus $y(a_0)=y(0)=q_0$.
We will now do the analysis of the behavior of $y$ over $[a_n,a_{n+1}]$ by making a case distinction between
$[a_n,a_n+1]$ and $[a_n+1,a_{n+1}]$. Note that for all $t$, $R(t)\geqslant|r(t)|\geqslant 0$.

\noindent \textbf{When $\mathbf{t\in[a_n,a_n+\tfrac{1}{2}]}$}, we have that
\[|\bitgen(\gamma_b,\mu,t)-b_{a_n}|\leqslant e^{-\mu}\]
but $b_{a_n}=1$ by definition thus
\[\bitgen(\gamma_b,\phi(t)+R(t),t)\geqslant 1-e^{-\phi(t)}\geqslant\tfrac{1}{2}\]
since $\phi(t)\geqslant\fastgen(\lambda_a;t)\geqslant 1$.
Furthermore,
\[\phi(t)\geqslant\fastgen(\lambda_a;t)\geqslant\fastgen(\lambda;n)\geqslant x_n\geqslant 2.\]
Thus $\psi(t)\phi(t)=2\bitgen(\gamma_b,\phi(t)+R(t),t)\phi(t)\geqslant x_n\geqslant 2$.
Furthermore,
\begin{equation}\label{eq:pwcgen:bound_g}
|g(t)-q_n|=|\dygen(\alpha_q,\beta_q;t)-q_n|\leqslant 2^{-x_n}
\end{equation}
thus we can apply Theorem~\ref{th:pereach} to get the existence of $y$ and item~\ref{th:pereach:conv} to get that
\begin{equation}\label{eq:pwcgen:bound_y_half}
|y(a_n+\tfrac{1}{2})-q_n|\leqslant 2^{-x_n}+e^{-x_n}
    \leqslant 2^{-x_n+1}.
\end{equation}
Note that \eqref{eq:pwcgen:bound_g} implies that
\[q_n-2^{-x_n}\leqslant g(t)\leqslant q_n+2^{-x_n}\]
and thus \eqref{eq:pwcgen:bound_y_half} proves that
\begin{equation}\label{eq:pwcgen:bounds_qn}
y(a_n+\tfrac{1}{2})-2^{-x_n+1}\leqslant q_n\leqslant y(a_n+\tfrac{1}{2})+2^{-x_n+1}.
\end{equation}
Furthermore, Theorem~\ref{th:pereach} item~\ref{th:pereach:lower_bound} also gives us that
\begin{align*}
\min(y(a_n),q_n-2^{-x})&\leqslant y(t)\leqslant\max(y(a_n),q_n+2^{-x_n})\\
\min(y(a_n),y(a_n+\tfrac{1}{2})-3\cdot2^{-x_n})&\leqslant y(t)\leqslant\max(y(a_n),y(a_n+\tfrac{1}{2})+3\cdot2^{-x_n})&&\text{using \eqref{eq:pwcgen:bounds_qn}}\\
&y(t)\in[y(a_n),y(a_n+\tfrac{1}{2})]+[-3\cdot2^{-x_n},3\cdot2^{-x_n}]\\
&y(t)\in[y(a_n),y(a_n+\tfrac{1}{2})]+[-2^{-\dysize(q_n)},2^{-\dysize(q_n)}]\numberthis\label{eq:pwcgen:ineq_y_2}.
\end{align*}

\noindent \textbf{When $\mathbf{t\in[a_n+k+\tfrac{1}{2},a_n+k+1]}$ for $\mathbf{a_n\leqslant a_n+k<a_{n+1}}$},
we have that
\[y'(t)=\psi(t)\pereach(t,\phi(t),y(t),g(t))\]
where $|\psi(t)|\leqslant 2$ since $|\bitgen|\leqslant 1$ by Theorem~\ref{th:bitgen}
and
\[\phi(t)\geqslant t+\fastgen(\lambda_a;t)\geqslant t+\fastgen(\lambda;n)\geqslant t+x_n.\]
Thus by Theorem~\ref{th:pereach} item~\ref{th:pereach:constant} we have that
\[|y(t)-y(a_n+k+\tfrac{1}{2})|
    \leqslant\int_{a_n+k+\tfrac{1}{2}}^t\psi(u)\exp(-\phi(u))du
    \leqslant\tfrac{1}{2}2e^{-x_n}\int_{a_n+k+\tfrac{1}{2}}^te^{-u}du.\]

\noindent \textbf{When $\mathbf{t\in[a_n+k,a_n+k+\tfrac{1}{2}]}$ for $\mathbf{a_n<a_n+k<a_{n+1}}$},
we have that
\[|\bitgen(\gamma_b,\mu,t)-b_{a_n+k}|\leqslant e^{-\mu}\]
but $b_{a_n+k}=0$ by definition thus
\[|\bitgen(\gamma_b,\phi(t)+R(t),t)|\leqslant e^{-\phi(t)-R(t)}.\]
Furthermore,
\[\phi(t)\geqslant t+\fastgen(\lambda_a;t)\geqslant t+\fastgen(\lambda;n)\geqslant t+x_n.\]
Thus
\[|\psi(t)|=2|\bitgen(\gamma_b,\phi(t)+R(t),t)|
    \leqslant 2e^{-t-x_n}e^{-R(t)}
    \leqslant 2e^{-t-x_n}e^{-|r(t)|}.\]
It follows that
\[y'(t)=\psi(t)r(t)\]
where
\[|\psi(t)r(t)|
    \leqslant 2e^{-\phi(t)-R(t)}|r(t)|
    \leqslant 2e^{-t-x_n}.\]
Consequently,
\[|y(t)-y(a_n+k)|
    \leqslant\int_{a_n+k}^{a_n+k+\tfrac{1}{2}}|\psi(u)r(u)|du
    \leqslant e^{-x_n}\int_{a_n+k}^{a_n+k+\tfrac{1}{2}}e^{-u}du.\]

\noindent \textbf{Putting everything together} we get that for all $t\in[a_n+\tfrac{1}{2},a_{n+1}]$,
\[|y(t)-y(a_n+\tfrac{1}{2})|\leqslant e^{-x_n}\int_{a_n+1}^te^{-u}du\leqslant e^{-x_n}\]
and thus using \eqref{eq:pwcgen:bound_y_half}, for all $t\in[a_n+\tfrac{1}{2},a_{n+1}]$,
\begin{equation*}\label{eq:pwcgen:ineq_y_1}
|y(t)-q_n|
    \leqslant e^{-x_n}+|y(a_n+\tfrac{1}{2})-q_n|
    \leqslant e^{-x_n}+2^{-x_n+1}\leqslant 2^{-x_n+2}\leqslant 2^{-\dysize(q_n)}.
\end{equation*}
Also recall \eqref{eq:pwcgen:ineq_y_2} that for all $t\in[a_n,a_n+\tfrac{1}{2}]$,
\begin{equation*}\label{eq:pwcgen:ineq_y_2bis}
y(t)\in[y(a_n),y(a_n+\tfrac{1}{2})]+[-2^{-\dysize(q_n)},2^{-\dysize(q_n)}].
\end{equation*}
We have already shown that the map
$Y(q_0,\alpha_q,\beta_q,\gamma_q,\lambda_a,t)=y(t)$ is uniformly-generable.
Finally, we need to show computability of the map $q\mapsto(\alpha_q,\beta_q,\gamma_b,\lambda_a)$.
Computability of $\alpha_q$ and $\beta_q$ follows from Theorem~\ref{th:dygen}. The sequence
$(a_n)_n$, and thus $(b_n)_n$, is easily computed from $q$. It follows that from
Theorem~\ref{th:bitgen} that $\gamma_b$ is computable from $b$, and from Theorem~\ref{th:fastgen}
that $\lambda_a$ is computable from $a$.
\end{proof}

\section{Proof of the main theorem}
\label{sec:mainth}

The proof works in several steps. First we show that using an almost constant function,
we can approximate functions that are bounded and change very slowly. We then
relax all these constraints until we get to the general case. In the following,
we only consider total functions over $\R$. See Remark on page~\pageref{rem:domain_of_def} for
more details.

\begin{defi}[Universality] Let $I\subseteq\R$ and $\mathcal{C}\subseteq C^0(I)\times C^0(I,\Rps)$. We say that
the universality property holds for $\mathcal{C}$ if there exists $d\in\N$, $\Gamma\subseteq\R^d$
and a uniformly-generable function $\mathfrak{u}:\Gamma\times I\to\R$ such that for every $(f,\varepsilon)\in\mathcal{C}$,
there exists $\alpha\in\Gamma$ such that
\[|\mathfrak{u}(\alpha;t)-f(t)|\leqslant\varepsilon(t)
\qquad\text{for all }t\in I.\]
The universality property is said to be \emph{effective} if furthermore the map
$(f,\varepsilon)\mapsto\alpha$ is $([\rho\to\rho]^2,\rho^d)$-computable.
\end{defi}

\begin{lem}\label{lem:univ_class_1}
There exists a constant $c>0$ such that the universality property holds for all
$(f,\varepsilon)$ on $\Rp$ such that for all $t\in\Rp$:
\begin{itemize}
\item $\varepsilon$ is decreasing and $-\log_2\varepsilon(t)\leqslant c'+t$ for some constant $c'$,
\item $f(t)\in[0,1]$,
\item $|f(t)-f(t')|\leqslant c\varepsilon(t+1)$ for all $t'\in[t,t+1]$.
\end{itemize}
Furthermore, the universality property is effective for this class.
\end{lem}

\begin{proof}[Proof Sketch]
This is essentially a application of $\pwcgentext$ with a small twist.
Indeed the bound on $f$ guarantees that dyadic rationals are enough. The bound on the rate of
change of $f$ guarantees that a single dyadic can provide an approximation for
a long enough time. And the bound on $\varepsilon$ guarantees that we do not need
too many digits for the approximations.
\end{proof}

\begin{proof}

Let $c=\tfrac{1}{8}$. Apply Theorem~\ref{th:pwcgen} to get $p,\delta\in\N$ and $\pwcgentext$. Let $f$ and $\varepsilon$
be as described in the statement.
For any $n\in\N$, let $q_n\in\D$ be such that
\begin{equation}\label{eq:len:univ_class_1:def_q}
|f(n)-q_n|
    \leqslant2^{-\dysize(q_n)}
    \leqslant c\varepsilon(n+1).
\end{equation}
Since by assumption, $-\log_2\varepsilon(n+1)\leqslant c'+n+1$, we can always choose $q_n$ so
that
\[\dysize(q_n)=\myceil{c'}+n+1.\]
Then by Theorem~\ref{th:pwcgen},
there exists $\alpha_q\in\R^p$ such that
\[|\pwcgen(\alpha_q;t)-q_n|\leqslant 2^{-\dysize(q_n)}\qquad\text{for any }t\in [a_n+\tfrac{1}{2},a_{n+1}]\]
and
\begin{equation}\label{eq:univ_class_1:range}
\pwcgen(\alpha_q;t)\in\left[\pwcgen(\alpha_q;a_n),\pwcgen(\alpha_q;a_n+\tfrac{1}{2})\right]
\qquad\text{for any }t\in [a_n,a_n+\tfrac{1}{2}]
\end{equation}
where
\[
a_n
    =\sum_{k=0}^{n-1}(\delta+\dysize(q_k))
    =\sum_{k=0}^{n-1}(\delta+\myceil{c'}+n+1)
    =n(\delta+\myceil{c'}+1)+\tfrac{1}{2}n(n-1).
\]
Introduce the function
\[\xi(t)=t(\delta+\myceil{c'}+1)+\tfrac{1}{2}t(t-1)\]
so that $a_n=\xi(n)$ and note that $\xi$ is increasing and generable.
Let $t\in[\xi^{-1}(a_n+\tfrac{1}{2}),\xi^{-1}(a_{n+1})]$, since $\xi$ is increasing, so is $\xi^{-1}$
and
\[n=\xi^{-1}(a_n)\leqslant\xi^{-1}(a_n+\tfrac{1}{2})\leqslant t\leqslant\xi^{-1}(a_{n+1})=n+1.\]
So in particular,
\[|f(n)-f(t)|\leqslant c\varepsilon(n+1)\]
by the assumption on $f$, since $t\in[n,n+1]$. It follows that
\begin{align*}
|\pwcgen(\alpha_q;\xi(t))-f(t)|
    &\leqslant|\pwcgen(\alpha_q;\xi(t))-q_n|\\
    &\hspace{.3cm}+|q_n-f(n)|+|f(n)-f(t)|\\
    &\leqslant 2^{-\dysize(q_n)}+2^{-\dysize(q_n)}+c\varepsilon(n+1)&&\text{since }\xi(t)\in[a_n+\tfrac{1}{2},a_{n+1}]\\
    &\leqslant 3c\varepsilon(n+1)\\
    &\leqslant 3c\varepsilon(t)&&\hspace{-3.5em}\text{since $\varepsilon$ is decreasing}\\
    &\leqslant\varepsilon(t)&&\text{since }3c\leqslant1.
\end{align*}
So in particular, in implies that for all $n\in\N$,
\begin{equation}\label{eq:univ_class_1:point_a_n}
|\pwcgen(\alpha_q;a_n+\tfrac{1}{2})-f(\xi^{-1}(a_n+\tfrac{1}{2}))|\leqslant 3c\varepsilon(n+1)
\end{equation}
and
\begin{equation}\label{eq:univ_class_1:point_a_n_half}
|\pwcgen(\alpha_q;a_{n+1})-f(n+1)|\leqslant 3c\varepsilon(n+2).
\end{equation}
Let $t\in[\xi^{-1}(a_{n+1}),\xi^{-1}(a_{n+1}+\tfrac{1}{2})]$, it follows using \eqref{eq:univ_class_1:range}
that there exists $\lambda\in[0,1]$ such that
\[\pwcgen(\alpha_q;\xi(t))=\lambda\pwcgen(\alpha_q;a_{n+1})+(1-\lambda)\pwcgen(\alpha_q;a_{n+1}+\tfrac{1}{2}).\]
We also have that
\[n+1=\xi^{-1}(a_{n+1})\leqslant t\leqslant \xi^{-1}(a_{n+1}+\tfrac{1}{2})\leqslant\xi^{-1}(a_{n+2})=n+2.\]
Thus
\begin{align*}
|\pwcgen(\alpha_q;\xi(t))-f(t)|
    &\leqslant \lambda|\pwcgen(\alpha_q;a_{n+1})-f(t)|\\
    &\hspace{2em}+(1-\lambda)|\pwcgen(\alpha_q;a_{n+1}+\tfrac{1}{2})-f(t)|\\
    &\leqslant |\pwcgen(\alpha_q;a_{n+1})-f(n+1)|+|f(n+1)-f(t)|\\
    &\hspace{2em}+|\pwcgen(\alpha_q;a_{n+1}+\tfrac{1}{2})-f(\xi^{-1}(a_{n+1}+\tfrac{1}{2}))|\\
    &\hspace{2em}+|f(\xi^{-1}(a_{n+1}+\tfrac{1}{2}))-f(t)|\\
    &\leqslant 3c\varepsilon(n+2)+|f(n+1)-f(t)|
        &&\hspace{-3em}\text{using \eqref{eq:univ_class_1:point_a_n}}\\
    &\hspace{2em}+3c\varepsilon(n+2)
        &&\hspace{-3em}\text{using \eqref{eq:univ_class_1:point_a_n_half}}\\
    &\hspace{2em}+|f(\xi^{-1}(a_{n+1}+\tfrac{1}{2}))-f(t)|\\
    &\leqslant 3c\varepsilon(n+2)+c\varepsilon(n+2)
        &&\hspace{-6em}\text{since }t\in[n+1,n+2]\\
    &\hspace{2em}+3c\varepsilon(n+2)\\
    &\hspace{2em}+c\varepsilon(t)
        &&\hspace{-4cm}\text{since }\xi^{-1}(a_{n+1}+\tfrac{1}{2})\in[t,t+1]\\
    &\leqslant 7c\varepsilon(n+2)+c\varepsilon(t)\\
    &\leqslant8c\varepsilon(t)&&\hspace{-2cm}\text{since $\varepsilon$ decreasing}\\
    &\leqslant\varepsilon(t)
        &&\hspace{-3em}\text{since }8c\leqslant1.
\end{align*}
Putting everything together, we can get that
\[|\pwcgen(\alpha_q;\xi(t))-f(t)|\leqslant\varepsilon(t)\qquad\text{for all }t\geqslant\xi^{-1}(a_0+\tfrac{1}{2}).\]
But note that $\xi^{-1}(a_0+\tfrac{1}{2})\leqslant\xi^{-1}(a_1)=1$ so we have that
\[|\pwcgen(\alpha_q;\xi(t))-f(t)|\leqslant\varepsilon(t)\qquad\text{for all }t\geqslant 1\]
and note that $(\alpha;t)\mapsto\pwcgen(\alpha;\xi(t))$ is uniformly-generable.

Note that this is not exactly the claimed result since it is only true for $t\geqslant 1$
instead of $t\geqslant 0$ but this can remedied for with proper shifting. Indeed,
consider the operator
\[(Sf)(t)=f(\max(t-1,0)).\]
We claim that if $(f,\varepsilon)$ satisfies the assumption of the Lemma, then so
does $(Sf,S\varepsilon)$ and
\begin{align*}
|\pwcgen(\alpha;\xi(t))-(Sf)(t)|
    &\leqslant(S\varepsilon)(t)&&\text{for all }t\geqslant 1\\
|\pwcgen(\alpha;\xi(t))-f(t-1)|
    &\leqslant\varepsilon(t-1)&&\text{for all }t\geqslant 1\\
|\pwcgen(\alpha;\xi(t+1))-f(t)|
    &\leqslant\varepsilon(t)&&\text{for all }t\geqslant 0.
\end{align*}
We need to show computability of the map $(f,\varepsilon)\mapsto\alpha_q$. By Theorem~\ref{th:pwcgen},
it is enough to show computability of $(f,\varepsilon)\mapsto q$. Since the continuous
function evaluation map is computable (for the representation we use),
the maps $(f,n)\mapsto f(n)$ and $(\varepsilon,n)\mapsto c\varepsilon(n+1)$ are
$([[\rho\to\rho],\nu_\N],\rho)$-computable. It follows that for every $n$, we can compute an integer
$p_n$ such that $2^{-p_n}\leqslant c\varepsilon(n+1)$. Indeed, $c\varepsilon(n+1)>0$ thus such a $p_n$
exists and any Cauchy sequence for $c\varepsilon(n+1)$ is eventually positive; therefore it suffices
to compute rational approximations of $c\varepsilon(n+1)$ with increasing precision until we get
a positive one, from which we can compute $p_n$. Given such a $p_n$, one can compute a dyadic
approximation $q_n$ of $f(n)$ with precision $p_n$. This sequence $(q_n)_n$ then satisfies
\eqref{eq:len:univ_class_1:def_q}.
\end{proof}

\begin{lem}\label{lem:univ_class_2}
The universality property holds for all $(f,\varepsilon)$ on $\Rp$ such that
$f$ and $\varepsilon$ are differentiable, $\varepsilon$ is decreasing and $f(t)\in[0,1]$ for all $t\in\Rp$.
Furthermore, the universality property is effective for this class if we are given
a representation of $f'$ and $\varepsilon'$ as well\footnote{In other words, the map $(f,f',\varepsilon)\mapsto\alpha$
    is $([\rho\to\rho]^4,\rho^d])-$computable. This is necessary because one can build some
    computable $f$ such that $f'$ is not computable \cite{Wei00}.}
\end{lem}

\begin{proof}[Proof Sketch]
Consider $F=f\circ h^{-1}$ and $E=\varepsilon\circ h^{-1}$ where $h$ is a fast-growing
function like $\fastgentext$. Then the faster $h$ grows, the slower $E$ and $F$
change and thus we can apply Lemma~\ref{lem:univ_class_1} to $(F,E)$. We
recover an approximation of $f$ from the approximation of $F$.
\end{proof}

\begin{proof}

Apply Lemma~\ref{lem:univ_class_1} to get $c>0$ and $\mathfrak{u}$ uniformly-generable.
For every $n\in\N$, let
\[a_n=\max\left(\frac{\max_{u\in[0,n+2]}|f'(u)|}{c\varepsilon(n)},\frac{-\min_{u\in[0,n+1]}\varepsilon'(u)}{\ln(2)\varepsilon(n)}\right)-1.\]
Check that $a_n$ is increasing because $\varepsilon$ is decreasing.
Then apply Theorem~\ref{th:fastgen} to get $\alpha_a$. Recall that $\fastgen(\alpha_a;\cdot)$ is positive,
thus we can let
\[g(\alpha;t)=\int_0^t1+\fastgen(\alpha;u)du,\qquad h_\alpha(t)=g(\alpha,t).\]
Clearly $g$ is uniformly-generable since $\fastgen$ is uniformly-generable.
Since $\fastgen$ is increasing, $h_\alpha$ is increasing. Furthermore,
\[h_{\alpha_a}(n+1)
    \geqslant\int_n^{1+n}\fastgen(\alpha_a;u)du
    \geqslant\int_n^{1+n}\fastgen(\alpha_a;n)du
    =\fastgen(\alpha_a;n)
    \geqslant a_n.
\]
Thus $h_{\alpha_a}(n)\rightarrow+\infty$ as $m\rightarrow+\infty$. This implies
that $h_{\alpha_a}$ is bijective from $\Rp$ to $\Rp$. Note that since $h_{\alpha_a}$ is
increasing then $h_{\alpha_a}^{-1}$ is also increasing. Also since $h_{\alpha_a}(t)\geqslant t$
then $h_{\alpha_a}^{-1}(t)\leqslant t$ for all $t\in\Rp$.
Let $f,\varepsilon$ be as described in the statement.
For any $\xi\in\Rp$, let
\[F(\xi)=f(h_{\alpha_a}^{-1}(\xi)),\qquad E(\xi)=\varepsilon(h_{\alpha_a}^{-1}(\xi)).\]
Then for any $t\in\Rp$,
\begin{align*}
F'(h_{\alpha_a}(t))
    &=(h_{\alpha_a}^{-1})'(h_{\alpha_a}(t))f'(h_{\alpha_a}^{-1}(h_{\alpha_a}(t)))\\
    &=\tfrac{1}{h_{\alpha_a}'(t)}f'(t)
        &&\text{since }h'(h^{-1})'\circ h=1\\
    &=\tfrac{1}{1+\fastgen(\alpha_a;t)}f'(t).
\end{align*}
Also note that since $h_{\alpha_a}'(t)=1+\fastgen(\alpha_a;t)\geqslant1$,
then $(h_{\alpha_a}^{-1})'(t)\leqslant1$ and thus $h_{\alpha_a}^{-1}$ is 1-Lipschitz.
Let $\xi\in\Rp$ and $\xi'\in[\xi,\xi+1]$. Write $\xi=h_{\alpha_a}(t)$ and $\xi'=h_{\alpha_a}(t')$, then
\begin{align*}
\tfrac{|F(\xi)-F(\xi')|}{E(\xi+1)}
    &\leqslant\tfrac{|\xi-\xi'|\max_{u\in[\xi,\xi']}|F'(\xi)|}{E(\xi+1)}\\
    &\leqslant\tfrac{\max_{u\in[t,t']}|F'(h_{\alpha_a}(u))|}{E(\xi+1)}
        &&\text{since }|\xi-\xi'|\leqslant 1\\
    &=\max_{u\in[t,t']}\tfrac{|f'(u)|}{E(\xi+1)(1+\fastgen(\alpha_a;u))}\\
    &\leqslant\tfrac{\max_{u\in[t,t+1]}|f'(u)|}{\varepsilon(h_{\alpha_a}^{-1}(h_{\alpha_a}(t)+1))(1+\fastgen(\alpha_a;t))}
        &&\text{since $\fastgen$ is increasing}.
    \intertext{but since $h_{\alpha_a}^{-1}$ is 1-Lipschitz and increasing,
        $h_{\alpha_a}^{-1}(h_{\alpha_a}(t)+1)\leqslant h_{\alpha_a}^{-1}(h_{\alpha_a}(t))+1=t+1$,}
    &\leqslant\tfrac{\max_{u\in[t,t+1]}|f'(u)|}{\varepsilon(t+1)(1+\fastgen(\alpha_a;t))}
        &&\text{since $\varepsilon$ is decreasing}\\
    &\leqslant\tfrac{\max_{u\in[t,t+1]}|f'(u)|}{\varepsilon(t+1)(1+\fastgen(\alpha_a;\lfloor t\rfloor))}
        &&\text{since $\fastgen$ is decreasing}\\
    &\leqslant\tfrac{\max_{u\in[t,t+1]}|f'(u)|}{\varepsilon(t+1)(1+a_{\lfloor t\rfloor}))}\\
    &\leqslant\tfrac{\max_{u\in[t,t+1]}|f'(u)|}{\varepsilon(t+1)\frac{\max_{u\in[0,\lfloor t\rfloor+2]}|f'(u)|}{c\varepsilon(\lfloor t\rfloor)}}\\
    &=c\tfrac{\max_{u\in[t,t+1]}|f'(u)|}{\max_{u\in[0,\lfloor t\rfloor+2]}|f'(u)|}
        \tfrac{\varepsilon(\lfloor t\rfloor)}{\varepsilon(t+1)}\\
    &\leqslant c&&\text{since $\varepsilon$ is decreasing.}
\end{align*}
Similarly,
\begin{align*}
E'(h_{\alpha_a}(t))
    &=(h_{\alpha_a}^{-1})'(h_{\alpha_a}(t))\varepsilon'(h_{\alpha_a}^{-1}(h_{\alpha_a}(t)))\\
    &=\tfrac{1}{h_{\alpha_a}'(t)}\varepsilon'(t)\\
    &=\tfrac{1}{1+\fastgen(\alpha_a;t)}\varepsilon'(t).
\end{align*}
Thus
\begin{align*}
-\log_2E(\xi)
    &=-\tfrac{1}{\ln 2}\int_0^\xi \frac{E'(e)}{E(e)}de-\tfrac{1}{2}\log_2 E(0)\\
    &=-\tfrac{1}{2}\log_2 \varepsilon(0)+\tfrac{1}{\ln 2}\int_0^\xi \frac{-E'(e)}{E(e)}de\\
    &\leqslant-\tfrac{1}{2}\log_2 \varepsilon(0)+\tfrac{\xi}{\ln 2}\sup_{e\in[0,\xi]}\frac{-E'(e)}{E(e)}\\
    &\leqslant-\tfrac{1}{2}\log_2 \varepsilon(0)+
        \tfrac{\xi}{\ln 2}\sup_{t\in[0,h_{\alpha_a}^{-1}(\xi)]}\frac{-E'(h_{\alpha_a}(t))}{E(h_{\alpha_a}(t))}\\
    &\leqslant-\tfrac{1}{2}\log_2 \varepsilon(0)+
        \tfrac{\xi}{\ln 2}\sup_{t\in[0,h_{\alpha_a}^{-1}(\xi)]}\frac{-\varepsilon'(t)}{\varepsilon(t)(1+\fastgen(\alpha_a;t))}\\
    &\leqslant-\tfrac{1}{2}\log_2 \varepsilon(0)+
        \tfrac{\xi}{\ln 2}\sup_{t\in[0,\xi]}\frac{-\varepsilon'(t)}{\varepsilon(t)(1+\fastgen(\alpha_a;t))}&&\text{using }h_{\alpha_a}^{-1}(\xi)\leqslant\xi\\
    &\leqslant-\tfrac{1}{2}\log_2 \varepsilon(0)+
        \tfrac{\xi}{\ln 2}\sup_{t\in[0,\xi]}\frac{-\varepsilon'(t)}{\varepsilon(t)(1+a_{\lfloor t\rfloor})}\\
    &\leqslant-\tfrac{1}{2}\log_2 \varepsilon(0)+
        \tfrac{\xi}{\ln 2}\sup_{t\in[0,\xi]}\frac{-\varepsilon'(t)}{\varepsilon(t)
            \frac{-\min_{u\in[0,\lfloor t\rfloor+1]}|\varepsilon'(u)|}{\ln(2)\varepsilon(\lfloor t\rfloor)}}\\
    &=-\tfrac{1}{2}\log_2 \varepsilon(0)+
        \xi\sup_{t\in[0,\xi]}\frac{\varepsilon(\lfloor t\rfloor)}{\varepsilon(t)}
            \frac{\varepsilon'(t)}{\min_{u\in[0,\lfloor t\rfloor+1]}\varepsilon'(u)}\\
    &\leqslant-\tfrac{1}{2}\log_2 \varepsilon(0)+
        \xi&&\hspace{-3cm}\text{since $\varepsilon$ is decreasing and $\varepsilon'$ is negative}
\end{align*}
and thus
\[-\log_2E(\xi)\leqslant c'+\xi\]
for some constant $c'$.
Therefore we can apply Lemma~\ref{lem:univ_class_1} to $(F,E)$ and get $\beta_{E,F}\in\R^p$ such that
\[|\mathfrak{u}(\beta_{E,F};\xi)-F(\xi)|\leqslant E(\xi)\qquad\text{for all }\xi\in\Rp.\]
For any $\alpha,\beta,t$, let
\[\bar{\mathfrak{u}}(\alpha,\beta;t)=\mathfrak{u}(\beta;g(\alpha,t)).\]
Clearly $\bar{\mathfrak{u}}$ is uniformly-generable because $\mathfrak{u}$ and $g$ are uniformly-generable.
Then for any $t\in\Rp$, recall that $g(\alpha_a;t)=h_{\alpha_a}(t)$ and thus
\[|\bar{\mathfrak{u}}(\alpha_a,\beta_{E,F};t)-f(t)|
    =|\mathfrak{u}(\beta;h_{\alpha_a}(t))-F(h_{\alpha_a}(t))|
    \leqslant E(h_{\alpha_a}(t))
    =\varepsilon(t).
\]
To show the effectiveness of the property, it suffices to show that $(f,f',\varepsilon)\mapsto(a,E,F)$
is computable. Indeed, $\alpha_a$ and $\beta_{E,F}$ are computable from $a,E,F$ by
Lemma~\ref{lem:univ_class_1} and Theorem~\ref{th:fastgen}. Given $a$, the maps $E$ and $F$
are computable from $f$ and $\varepsilon$ because $h_{\alpha_a}$ is computable and increasing, thus
its inverse is computable. Finally, to show computability of $a$, notice that to define a
suitable value for each $a_n$, it is enough to compute an upper bound on the maximum of continuous functions -- defined from
$f,f',\varepsilon,\varepsilon'$ -- over compact intervals, which is a computable operation.
\end{proof}

\begin{lem}\label{lem:univ_class_3}
The universality property holds for all $(f,\varepsilon)$ on $\Rp$ such that
$f$ and $\varepsilon$ are differentiable and $\varepsilon$ is decreasing.
Furthermore, the universality property is effective for this class if we are given
a representation of $f'$ and $\varepsilon'$ as well.
\end{lem}


\begin{proof}

Apply Lemma~\ref{lem:univ_class_2} to get $p\in\N$ and $\mathfrak{u}$ uniformly-generable.
Let $a\in\N^\N$ be an increasing sequence,
and apply Theorem~\ref{th:fastgen} to get $\alpha_a$. Recall that $\fastgen(\alpha_a;\cdot)$
is positive and increasing. Let $f,\varepsilon$ be as described in the statement.
For any $t\in\Rp$, let
\[F(t)=\tfrac{1}{2}+\tfrac{f(t)}{1+\fastgen(\alpha_a;t)},
\qquad E(t)=\tfrac{\varepsilon(t)}{1+\fastgen(\alpha_a;t)}.\]
Then for any $n\in\N$ and $t\in[n,n+1]$, we have that
\[|F(t)-\tfrac{1}{2}|
    =\tfrac{f(t)}{1+\fastgen(\alpha_a;t)}
    \leqslant \tfrac{|f(t)|}{1+a_n}.
\]
Thus we can choose $a_n=2\max_{u\in[t,t+1]}|f(u)|$ and get that $|F(t)-\tfrac{1}{2}|\leqslant\tfrac{1}{2}$
for all $t\in\Rp$, and thus $F(t)\in[0,1]$. Furthermore, $F$ is differentiable
and $E$ is decreasing because $\varepsilon$ is decreasing and $\fastgentext$ increasing.
Apply Lemma~\ref{lem:univ_class_2}
to $(F,\varepsilon)$ to get $\beta_F\in\R^p$ such that
\[|\mathfrak{u}(\beta_F;t)-F(t)|\leqslant E(t)\qquad\text{for all }t\in\Rp.\]
For any $\alpha,\beta,t$, let
\[\bar{\mathfrak{u}}(\alpha,\beta;t)=(1+\fastgen(\alpha;t))(\mathfrak{u}(\beta;t)-\tfrac{1}{2}).\]
Clearly $\bar{\mathfrak{u}}$ is uniformly-generable because $\mathfrak{u}$ and $\fastgentext$ are uniformly-generable.
Then for any $t\in\Rp$,
\begin{align*}
|\bar{\mathfrak{u}}(\alpha_a,\beta_F;t)-f(t)|
    &=\left|(1+\fastgen(\alpha_a;t))\left(\mathfrak{u}(\beta;h_{\alpha_a}(t))-\tfrac{1}{2}-\tfrac{f(t)}{1+\fastgen(\alpha_a;t)}\right)\right|\\
    &=\left(1+\fastgen(\alpha_a;t))|\left(\mathfrak{u}(\beta;h_{\alpha_a}(t))-F(t)\right)\right|\\
    &\leqslant (1+\fastgen(\alpha_a;t))E(t)\\
    &\leqslant \varepsilon(t).
\end{align*}
The effectiveness of $\alpha_a$ and $\beta_F$ comes from previous lemmas and boils down again
to compute an upper bound on the maximum of a continuous function.
\end{proof}

\begin{lem}\label{lem:univ_class_4}
The universality property holds for all continuous $(f,\varepsilon)$ on $\Rp$.
Furthermore, the universality property is effective for this class.
\end{lem}

\begin{proof}

Apply Lemma~\ref{lem:univ_class_3} to get $p\in\N$ and $\mathfrak{u}$ uniformly-generable.
Let $f\in C^0(\Rp,\R)$ and $\varepsilon\in C^0(\Rp,\Rps)$. Then there exists
$\tilde{f}\in C^1(\Rp,\R)$ and a decreasing $\tilde{\varepsilon}\in C^1(\Rp,\Rps)$
such that
\begin{equation}\label{eq:univ_class_3:def_f_eps_c1}
|\tilde{f}(t)-f(t)|\leqslant\tilde{\varepsilon}(t)\leqslant\tfrac{\varepsilon(t)}{2}.
\end{equation}
We can then apply Lemma~\ref{lem:univ_class_3} to $(\tilde{f},\tilde{\varepsilon})$
to get $\alpha_{\tilde{f}}\in\R^p$ such that
\[|\mathfrak{u}(\alpha_{\tilde{f}};t)-\tilde{f}(t)|\leqslant \tilde{\varepsilon}(t)\qquad\text{for all }t\in\Rp.\]
But then for any $t\in\Rp$,
\begin{align*}
|\mathfrak{u}(\alpha_{\tilde{f}};t)-f(t)|
    &\leqslant|\mathfrak{u}(\alpha_{\tilde{f}};t)-\tilde{f}(t)|+|\tilde{f}(t)-f(t)|\\
    &\leqslant\tilde{\varepsilon}(t)+\tilde{\varepsilon}(t)\\
    &\leqslant\varepsilon(t).
\end{align*}
To show the computability of $\alpha_{\tilde{f}}$, it suffices to show computability of $\tilde{f}$
and $\varepsilon$, and their derivatives, from $f$ and $\varepsilon$, and apply Lemma~\ref{lem:univ_class_3}.
Is it not hard to find $C^1$ functions satisfying \eqref{eq:univ_class_3:def_f_eps_c1}. For example,
one can proceed over all intervals $[n,n+1]$ and then use $C^1$ pasting. Over a compact interval,
one can use an effective variant of Stone-Weierstrass theorem.
\end{proof}

\begin{lem}\label{eq:univ_class_4_extend}
There exists a uniformly-generable function $\mathfrak{u}$
such that for all $f\in C^0(\Rp,\R)$, $\varepsilon\in C^0(\R,\Rps)$ and $\delta>0$,
there exists $\alpha$ such that
\begin{itemize}
\item $|\mathfrak{u}(\alpha;t)-f(t)|\leqslant \varepsilon(t)$ for all $t\geqslant 0$,
\item $|\mathfrak{u}(\alpha;t)|\leqslant |f(-t)|+\varepsilon(t)$ for all $t\in[-\delta,0]$,
\item $|\mathfrak{u}(\alpha;t)|\leqslant\varepsilon(t)$ for all $t\leqslant-\delta$.
\end{itemize}
Furthermore, the map $(f,\varepsilon,\delta)\mapsto\alpha$ is $([[\rho\to\rho]^2,\rho],\rho)-$computable.
\end{lem}

\begin{proof}
Apply Lemma~\ref{lem:univ_class_4} to get $\mathfrak{u}$. Note that $t\mapsto f(\sqrt{t})$,
$t\mapsto\varepsilon(\sqrt(t))$ and $\max$ are continuous, so there exists $\alpha$
such that
\[|\mathfrak{u}(\alpha;t)-f(\sqrt{t})|\leqslant\tfrac{1}{2}\min(\varepsilon(\sqrt{t}),\varepsilon(-\sqrt{t}))\qquad\text{for all }t\geqslant 0.\]
For any $\alpha,t$ define
\[\mathfrak{U}(\alpha,\beta,\delta;t)=s(t)\mathfrak{u}(\alpha;t^2),
\qquad s(t)=\tfrac{1}{2}+\tfrac{1}{2}\tanh(A(t)(\tfrac{\delta}{2}+t)),\qquad A(t)=\fastgen(\beta;t).\]
Let $a_n$ be a sequence such that for all $n\in\N$,
\[a_n\geqslant \tfrac{2}{\delta}\left(\sup_{t\in[n,n+1]}|f(t)|+\sup_{t\in[-n-1,n+1]}-\log\varepsilon(t)\right).\]
Then there exists $\beta_a$ such that for all $n\in\N$,
\[\fastgen(\beta;t)\geqslant a_n\qquad\text{for all }t\geqslant n.\]
Let $t\geqslant 0$, then
\begin{align*}
A(t)&=\fastgen(\beta;t)\\
    &\geqslant a_{\lfloor t\rfloor}\\
    &\geqslant \tfrac{2}{\delta}\left(\sup_{u\in[\lfloor t\rfloor,\lceil t\rceil]}|f(u)|+\sup_{u\in[-\lceil t\rceil,\lceil t\rceil]}-\log\varepsilon(u)\right)\\
    &\geqslant \tfrac{2}{\delta}\left(|f(t)|-\log\varepsilon(t)\right)\\
(\tfrac{\delta}{2}+t)A(t)
    &\geqslant|f(t)|-\log\varepsilon(t)\\
|1-\tanh((\tfrac{\delta}{2}+t)A(t))|
    &\leqslant e^{-|f(t)|+\log\varepsilon(t)}\\
|1-s(t)|
    &\leqslant \tfrac{\varepsilon(t)}{2}e^{-|f(t)|}.
\end{align*}
It follows that
\begin{align*}
|f(t)-\mathfrak{U}(\alpha,\beta,\delta;t)|
    &\leqslant |f(t)-\mathfrak{u}(\alpha;t^2)|+|\mathfrak{u}(\alpha;t^2)(1-s(t))|\\
    &\leqslant \tfrac{1}{2}\varepsilon(t)+|\mathfrak{u}(\alpha;t^2)||1-s(t)|\\
    &\leqslant \tfrac{1}{2}\varepsilon(t)+|\mathfrak{u}(\alpha;t^2)||1-s(t)|\\
    &\leqslant \tfrac{1}{2}\varepsilon(t)+(|f(t)|+\tfrac{1}{2}\varepsilon(t))\tfrac{\varepsilon(t)}{2}e^{-|f(t)|}\\
    &\leqslant\varepsilon(t)&&\text{using }e^{-x}x\leqslant 1.
\end{align*}
Let $t\in[-\delta,0]$, then $|s(t)|\leqslant 1$ thus
\[
|\mathfrak{U}(\alpha,\beta,\delta;t)|
    \leqslant|\mathfrak{u}(\alpha;t^2)|
    \leqslant|f(-t)|+\varepsilon(t).\]
Let $t\leqslant-\delta$, then with a similar argument as above
\[|s(t)|
    \leqslant \tfrac{\varepsilon(t)}{2}e^{-|f(-t)|}.\]
It follows that
\[
|\mathfrak{U}(\alpha,\beta,\delta;t)|
    \leqslant|\mathfrak{u}(\alpha;t^2)||s(t)|
    \leqslant(|f(-t)|+\varepsilon(t))\tfrac{\varepsilon(t)}{2}e^{-|f(-t)|}
    \leqslant\varepsilon(t).\]
The effectiveness of $\alpha$ comes from Lemma~\ref{lem:univ_class_4} and the fact the map
$(f,t)\mapsto f(\sqrt{t})$ is $([[\rho\to\rho],\rho],\rho)-$computable since $\sqrt{.}$ is
$(\rho,\rho)-$computable. To show the computability of $\beta$, it suffices to show the computability
of $a_n$, which boils down to computing an upper bound for $f$ and $\varepsilon$ over compact intervals.
The effectiveness for $\delta$ is trivial since it is the identity
mapping ($\delta$ is given unmodified to $\mathfrak{U}$).
\end{proof}

\begin{lem}\label{lem:univ_class_5}
The universality property holds for all continuous $(f,\varepsilon)$ on $\R$.
Furthermore, the universality property is effective for this class.
\end{lem}


\begin{proof}
Let $c=\tfrac{1}{4}$.
Apply Lemma~\ref{eq:univ_class_4_extend} to get $\mathfrak{u}$. Then there exists $\alpha$ such
that
\begin{itemize}
\item $|\mathfrak{u}(\alpha;t)-f(t)|\leqslant c\varepsilon(t)$ for all $t\geqslant 0$,
\item $|\mathfrak{u}(\alpha;t)|\leqslant |f(-t)|+c\varepsilon(t)$ for all $t\in[-1,0]$,
\item $|\mathfrak{u}(\alpha;t)|\leqslant c\varepsilon(t)$ for all $t\leqslant-1$.
\end{itemize}
For all $t\in\R$, let
\[g(t)=f(-t)-\mathfrak{u}(\alpha;-t).\]
Since this is a continuous function, we can apply the lemma again to get $\alpha'$
such that
\begin{itemize}
\item $|\mathfrak{u}(\alpha';t)-g(t)|\leqslant c\varepsilon(t)$ for all $t\geqslant 0$,
\item $|\mathfrak{u}(\alpha';t)|\leqslant |g(-t)|+c\varepsilon(t)$ for all $t\in[-1,0]$,
\item $|\mathfrak{u}(\alpha';t)|\leqslant c\varepsilon(t)$ for all $t\leqslant-1$.
\end{itemize}
For all $\alpha,\alpha',t$ let
\[\mathfrak{U}(\alpha,\alpha';t)=\mathfrak{u}(\alpha;u)+\mathfrak{u}(\alpha';-t).\]
We claim that $\mathfrak{U}$ satisfies the theorem:
\begin{itemize}
\item If $t\geqslant 1$ then $-t\leqslant-1$ thus
    \[|\mathfrak{U}(\alpha,\alpha';t)-f(t)|
        \leqslant|\mathfrak{u}(\alpha;t)-f(t)|+|\mathfrak{u}(\alpha';-t)|
        \leqslant c\varepsilon(t)+c\varepsilon(t)
        \leqslant\varepsilon(t).\]
\item If $0\leqslant t\leqslant 1$ then $-1\leqslant -t\leqslant 0$ thus
    \begin{align*}
    |\mathfrak{U}(\alpha,\alpha';t)-f(t)|
        &\leqslant|\mathfrak{u}(\alpha;t)-f(t)|+|\mathfrak{u}(\alpha';-t)|\\
        &\leqslant c\varepsilon(t)+|g(-t)|+c\varepsilon(t)\\
        &\leqslant c\varepsilon(t)+|g(-t)|+c\varepsilon(t)\\
        &\leqslant c\varepsilon(t)+|f(t)-\mathfrak{u}(\alpha;t)|+c\varepsilon(t)\\
        &\leqslant 3c\varepsilon(t)\leqslant\varepsilon(t).
    \end{align*}
\item If $t\leqslant 0$ then $-t\geqslant 0$ and thus
    \begin{align*}
    |\mathfrak{U}(\alpha,\alpha';t)-f(t)|
        &=|\mathfrak{u}(\alpha';t)+\mathfrak{u}(\alpha;t)-f(t)|\\
        &=|\mathfrak{u}(\alpha';-t)-g(-t)|\\
        &\leqslant c\varepsilon(t).
    \end{align*}
\end{itemize}
The computability of $\alpha$ follows directly from the previous lemma.
\end{proof}

We can now show the main theorem.


\begin{proof}[Proof of Theorem~\ref{th:universal_pivp}]
This is mostly rewriting but we do it full for completeness.

Apply Lemma~\ref{lem:univ_class_5} to get a uniformly-generable function $\mathfrak{u}:\Gamma\times\R\to\R$.
By definition of $\mathfrak{u}$, there exists an integer $d$, a polynomial matrix $q$
with coefficients in $\K$, $t_0\in\K$ and a computable function $y_0:\Gamma\to\R^d$
such that for all $\alpha\in\Gamma$ there exists
$y_\alpha:\R\rightarrow\R^d$ such that
\begin{itemize}
\item $y_\alpha(t_0)=y_0(\alpha)$ and $y_\alpha'(t)=q(y_\alpha(t))$ for all $t\in\R$,
\item $(y_\alpha(t))_1=\mathfrak{u}(\alpha;t)$ for all $t\in\R$.
\end{itemize}
Note that $q$ is a polynomial that does not depend on $\alpha$ but potentially has coefficients
in $\K$ that are not rational. We can eliminate those as explained in Remark~\ref{rem:unif_gen_coeff}.
Now we get that for all continuous functions $f,\varepsilon$, there exists $\alpha\in\Gamma$
such that $|f(t)-\mathfrak{u}(\alpha,t)|\leqslant\varepsilon(t)$. Therefore if we consider
the unique solution to $z(0)=y(\alpha;0)$ and $z'=q(z)$ then $z(t)=y_\alpha(t)$ and
we have the result. Note that we have not used the initial condition $z(t_0)=y_0(\alpha)$ because we
want $t_0=0$ in the statement of the theorem. Furthermore, the initial condition $y(\alpha;0)$ is computable from
$f$ and $\varepsilon$ because $y(\alpha;0)$ is computable from $t_0$ and $y_0(\alpha$) by
Proposition~\ref{prop:gen_implies_computable}, $t_0\in\K$ is computable,
$y_0$ is computable and $\alpha$ is computable from $f,\varepsilon$
by Lemma~\ref{lem:univ_class_5}.
\end{proof}

\bibliographystyle{alpha}
\bibliography{extracted}

\newpage
\appendix










\end{document}